\documentclass{amsart}

\newtheorem{theorem}{Theorem}[section]

\theoremstyle{definition}

\theoremstyle{remark}

\numberwithin{equation}{section}

\usepackage{amsmath,amsfonts,amssymb}
\usepackage{xcolor}
\usepackage[extra,safe]{tipa}
\usepackage[applemac]{inputenc}

\def\my_c{c_\infty}
\def\ui{{\underline{i}}}

\def\leftB{[\![}
\def\rightB{]\!]}
\def\I{{\rm{I\!\!I}}}
\def\leftB{[\![}
\def\rightB{]\!]}

\def\F{{\mathcal{ F}}}

\def\R{{\mathbb{R}} }
\def\D{{\mathbb{D}} }

\def\N{{\mathbb{N}} }
\def\E{{\mathbb{E}}  }

\def\P{{\mathbb{P}}  }

\def\I{{\mathbb{I}}}
\def\det{{\rm{det}}}
\def\Tr{{\rm{Tr}}}
\def\bint#1^#2{\displaystyle{\int_{#1}^{#2}}}
\def\bsum#1^#2{\displaystyle{\sum_{#1}^{#2}}}

\def\xdt_#1{X_#1(\Delta t)}

\newtheorem{THM}{Theorem}[section]
\newtheorem{PROP}[theorem]{Proposition}
\newtheorem{COROL}{Corollary}[section]
\newtheorem{LEMME}[theorem]{Lemma}
\newtheorem{REM}[theorem]{Remark}

\def\Tr{{{\rm Tr}}}

\def\0{{\mathbf{0}}}

\newcommand \A[1]{{\bf (#1)}}

\title[Weak Error for Fractional in Time P(I)DEs]{
Weak Error for Continuous Time Markov Chains 
 Related to Fractional in Time P(I)DE{s}
}

\author{M. Kelbert}
\address{National Research University Higher School of Economics, Shabolovka 31, Moscow, Russian Federation.}
\email{MKelbert@hse.ru}

\author{V. Konakov}
\address{National Research University Higher School of Economics, Shabolovka 31, Moscow, Russian Federation.}
\email{VKonakov@hse.ru}
\author{S. Menozzi}
\address{National Research University Higher School of Economics, Shabolovka 31, Moscow, Russian Federation and  LaMME, UMR CNRS 8070, Universit\'e d'Evry Val d'Essonne, 23 Boulevard de France, 91037 Evry, France.}
\email{stephane.menozzi@univ-evry.fr}


\date{\today}


\begin{document}
\maketitle

\begin{abstract}
We provide sharp error bounds for the difference between the transition densities of
some multidimensional Continuous Time Markov Chains (CTMC)
and the fundamental solutions of
some fractional in time Partial (Integro) Differential Equations (P(I)DEs). Namely, we consider equations involving a time fractional derivative of Caputo type  
and a spatial operator corresponding to the generator of  a non degenerate 
Brownian or stable driven Stochastic Differential Equation (SDE). 
\end{abstract}

\section{Introduction}

We are interested in the probabilistic approximation of P(I)DEs of the following type:
\begin{equation}
\label{CAUCHY_I}
\begin{cases}
\partial_t^\beta u(t,x)=Lu(t,x),\ (t,x)\in \R_+^*\times \R^d,\\
u(0,x)=f(x),\ x\in \R^d,
\end{cases}
\end{equation}
where $\partial_t^\beta, \beta \in (0,1)$, stands for the Caputo--Dzherbashyan derivative and $L $ is the generator of a Brownian or stable driven non degenerate SDE (see equations \eqref{CAPUTO}, \eqref{SPATIAL_OP} below for the respective definitions of the Caputo-Dzherbashyan derivative and the spatial operators considered). Equations of the previous type appear in many applicative fields from natural sciences to finance, see e.g. Meerschaert \textit{et al.}, \cite{meer:siko:12}, \cite{meer:stra:13} and references therein.

Under suitable assumptions, the solution of \eqref{CAUCHY_I} can be represented as $u(t,x):=\E[f(X_{Z_t^\beta}^{0,x})] $ where $(X_s^{0,x})_{s\ge 0} $ solves the SDE with generator $L$ and $(Z_t^\beta)_{t\ge 0} $ is the inverse of a stable subordinator of index $\beta\in (0,1)$ independent from $X^{0,x}$. This therefore extends the {``}usual'' Feynman-Kac representation, corresponding to $\beta=1 $, to the fractional case $\beta\in (0,1) $.

Many probabilistic numerical approximations of $u(t,x)$ have been considered when $\beta=1 $. We can for instance mention the works of Konakov \textit{et al.} (see \cite{kona:mamm:00}, \cite{kona:mamm:02}, \cite{kona:mamm:09} in the non degenerate diffusive case or \cite{kona:meno:10} for  SDEs driven by symmetric stable processes) that 
investigate the Euler scheme or, more generally, the Markov Chain approximation of the spatial motions in terms of Edgeworth expansions or Local Limit Theorems (LLTs). We can also refer to the works of Bally and Talay for the Euler scheme of some hypoelliptic diffusions \cite{ball:tala:96:1}, \cite{ball:tala:96:2} or \cite{kona:meno:molc:10} for associated LLTs. 
 
  In the current (strictly) fractional framework two additional difficulties appear. Firstly, it is known that the fundamental solutions to \eqref{CAUCHY_I} exhibit, additionally to the \textit{usual} time singularity in short time, a diagonal spatial singularity, see e.g. Eidelman and Kochubei \cite{koch:eide:04} for the current case
  or Kochubei \cite{koch:14}
for extensions to higher order fractional derivatives $\beta\in(1,2)$.
  Secondly, the inverse of the subordinator leads to consider random integration times that might be either very small or very long. The analysis of the discretization error in the previously mentioned works was always performed for a fixed  final time horizon $T$ and the constants controlling the  error estimates depend explosively on $T$ for small and large times. We must therefore carefully control these explosions. 
In short time we must handle the spatial singularity of the fractional heat kernel whereas in long time, the explosion is compensated by the fast decay of the density of the inverse subordinator. 
  Let us mention that those difficulties may deteriorate the {``}usual" convergence rate for the weak error even for the Euler approximation. 
  We establish error bounds for the Euler scheme and the Markov Chain approximation of a diffusive SDE, Theorem \ref{THM_DIFF}, and for the Euler scheme of a stable driven SDE, Theorem \ref{THM_STABLE}. 
  We emphasize that Kolokoltsov \cite{kolo:09} also considered probabilistic schemes to approximate equations of type \eqref{CAUCHY_I} in a more general setting, namely considering a fractional like time derivative that could depend on space as well. He establishes convergence of the schemes towards the expected solution but since the approach relies on semigroup techniques, no convergence rates are provided.

The article is organized as follows. We first recall in Section  \ref{PROB_TOOLS} which are the probabilistic tools needed to give representations and approximations of the solution to equation   \eqref{CAUCHY_I}. We then state our main results in Section \ref{MAIN_RESULTS}. The proofs are given in Section \ref{PROOF} which is the technical core of the work. Some perspectives are considered in Section \ref{PERP}. Technical results concerning the parametrix expansions, which are needed to establish the short and long time behaviour of the involved spatial densities, are collected in Appendix \ref{ASYMP_SMALL_TIMES}. As a by-product of our analysis, we obtain two-sided heat kernel estimates for the fundamental solution of \eqref{CAUCHY_I} in Appendix \ref{APP_HK}.
   

\section{Probabilistic objects associated with Cauchy Problems with Time Fractional Derivatives.}
\label{PROB_TOOLS}

We first recall how a probabilistic representation for the solution of \eqref{CAUCHY_I} can be derived considering the special 
case of $L $ being associated with a $d$-dimensional symmetric stable process of index $\alpha\in (0,2] $, thus including the Brownian motion. We write for $\phi \in C_0^2(\R^d,\R) $ (space of twice continuously differentiable functions with compact support),
 $L\phi(x)=\frac 12 \Delta \phi(x) $ for $\alpha=2 $ whereas for $\alpha \in (0,2) $: 
\begin{eqnarray}
\label{STABLE_GEN}
L\phi(x)&=&\int_{\R^d} \{\phi(x+y )-\phi(x)-\nabla \phi(x)\cdot y \I_{|y|\le 1}\} \mu(dy)\nonumber\\
&=&\int_{\R^+\times S^{d-1}}\{\phi(x+|y|\bar y )-\phi(x)-\nabla \phi(x)\cdot |y|\bar y \I_{|y|\le 1})\} |y|^{-(1+\alpha)}d|y| \nu(d\bar y),
\end{eqnarray}
where $ \nu$ is the spherical part of $\mu $ and is assumed to be non degenerate, i.e. there exists 
\begin{equation}
\label{ND_SPEC_MEAS}
\Lambda\ge 1,\ \forall \xi \in \R^d,\ \Lambda^{-1}|\xi|^\alpha\le \int_{S^{d-1}} |\langle \xi,\eta \rangle|^\alpha \nu(d\eta)\le \Lambda|\xi|^\alpha.
\end{equation}
Observe that if $(S_u^\alpha)_{u\ge 0} $ is a stable process with generator $L $ then
$$\forall t\ge 0,\ \forall \xi \in \R^d,\ \varphi_{S_t^\alpha}(\xi):=\E[\exp(i\langle \xi,S_t^\alpha\rangle )]:= \exp\left(-tC_{\alpha}\int_{S^{d-1}} |\langle \xi,\eta \rangle|^\alpha \nu(d\eta) \right), $$
where the explicit value of $C_{\alpha} $ can be found in Sato \cite{sato:05}.

Let us now discuss heuristically how some suitable scaling limits of 
Continuous Time Random Walks (CTRWs) actually give a probabilistic interpretation of the solution to \eqref{CAUCHY_I}.  
These objects thus provide a natural approximation scheme to \eqref{CAUCHY_I}.
Basically, CTRWs are random walks that wait a certain amount of time between their jumps. In the i.i.d. case, the mechanism can be described as follows: let $(T_i,Y_i)_{i\in \N^*} $ be a sequence of $\R^+\times \R^d $-valued i.i.d pairs of random variables defined on some probability space $(\Omega,\F,\P) $.
For all $t\ge 0$ one defines:
\begin{eqnarray*}
\Gamma_t:=\sum_{i=0}^{N_t} Y_i,\ N_t:=\max\{m\in \N:\sum_{i=1}^m T_i\le t \}.
\end{eqnarray*}
If the $(T_i)_{i\in\N^*} $ are exponentially distributed and independent  of the $(Y_i)_{i\in \N^*}$ the process $(\Gamma_t)_{t\ge 0}$ is easily shown to be  Markovian.
This property clearly fails in the general case. Of particular interest is the case when the $(T_i)_{i\in \N^*} $ and the $(Y_i)_{i\in \N^*}$ are independent and respectively belong to the domain of attraction of a $\beta $-stable and $\alpha $-stable law with $\beta \in(0,1) $ and $\alpha \in (0,2] $. This property indeed yields that : 
\begin{equation}
\label{CV_LAW_APP}
(n^{-1/\beta}\sum_{i=1}^{\lfloor nu\rfloor } T_{i})_{u\ge 0} \Rightarrow (S_u^{\beta,+})_{u\ge 0}, \ \ {\rm and} \ \ (n^{-1/\alpha}\sum_{i=1}^{\lfloor nu \rfloor} Y_{i})_{u\ge 0} \Rightarrow (S_u^\alpha)_{u\ge 0},
\end{equation}
where $\lfloor \cdot\rfloor $ stands for the integer part and $ \Rightarrow$  indicates the usual convergence in law for processes. We also used the convention that $\sum_{i=1}^0 =0$. In \eqref{CV_LAW_APP}, $S^{\beta,+} $ is a $\beta $-stable subordinator, i.e. a L\'evy process with positive jumps and its Laplace transform writes $\psi_{S_u^{\beta,+}}(\lambda)=\E[\exp(-\lambda S_u^{\beta,+})]=\exp(- u \lambda^\beta),\ \lambda \ge 0 $. On the other hand, we will assume that $S^\alpha$ is a symmetric $\R^d$-valued stable process with generator as in \eqref{STABLE_GEN}.


Now, the rescaled process associated with the number of jumps $(N_t)_{t\ge 0}$ also has a limit, namely $(n^{-\beta} N_{nt})_{t\ge 0} \rightarrow (Z_t^\beta)_{t\ge 0} $, where $Z_t^\beta:=\inf\{s\ge 0: S_s^{\beta,+}> t \} $ which is the inverse process of a $ \beta$-stable subordinator. Since $S_s^{\beta,+} $ is increasing in $s$, $Z_t^\beta $ also corresponds to the first passage time of $S^{\beta,+} $ above the level $t $.
Thus, one  formally has:
$$\begin{array}{l}
\left(n^{-\beta /\alpha}\Gamma_{{nt}}\right)_{t\ge 0}= \left((n^{\beta})^{-1/\alpha}\sum\limits_{i=1}^{N_{nt}}Y_i\right)_{t\ge 0}\\
=\left((n^{\beta})^{-1/\alpha}\sum\limits_{i=1}^{n^{\beta}(n^{-\beta}N_{nt})}Y_i\right)_{t\geq 0}\Rightarrow 
\left( S_{Z_t^\beta}^\alpha\right)_{t\ge 0}.
\end{array}$$
From the independence of $S^\alpha$ and $Z^\beta$,  the limit random variable $ S_{Z_t^\beta}^\alpha$ has, for any $t>0$, an \textit{explicit} density that writes:
\begin{eqnarray}
\label{DENS_GLOB}  
q(t,x)=\int_0^{+\infty}p_{S^\alpha}(u,x) p_{Z^\beta}(t,u)  du,
\end{eqnarray}
where $p_{S^\alpha}(u,x)=\frac 1{(2\pi)^d}\int_{\R^d}\exp(-i\langle x, \lambda \rangle) \varphi_{S_u^\alpha}(\lambda)d\lambda $ and $p_{Z^\beta}(t,u) $ stands for the density of  $Z_t^\beta$ at point $u $. Observe that since $\P[Z_t^\beta\le u]=\P[S_u^{\beta,+}> t] $ we have the following relation between the density  of the inverse $Z^\beta$ and the density of the subordinator $S^{\beta,+} $:
\begin{equation}
\label{DENS_INV_SUB}
p_{Z^\beta}(t,u)=\partial_u \P[Z_t^\beta\le u]=\partial_u(1-\P[S_u^{\beta,+}\le t])=-\partial_u\int_0^{t}p_{S^{\beta,+}}(u,y)dy.
\end{equation}
We also have the following important equation for this density.
\begin{PROP}[PDE associated with $p_{Z^\beta} $]
\label{PROP_DENS_P_Z}
For fixed $t>0$, the density $p_{Z^\beta}(t,u) $ satisfies in the distributional sense the equation:
\begin{equation}
\label{EQ_DENS_Z}
(\D_t^\beta +\partial_u) p_{Z^\beta}(t,u)=\frac{t^{-\beta}}{\Gamma(1-\beta)}\delta(u),
\end{equation}
where $\delta$ stands for the Dirac mass at 0 and $\D_t^\beta $ denotes the Riemann-Liouville derivative. Namely, for a function $h$ and $\beta\in (0,1) $: 
\begin{eqnarray*}
\D_t^\beta h(t)=\frac{1}{\Gamma(1-\beta)}\frac{d}{dt}\left\{\int_0^t h(t-u)u^{-\beta}du\right\}.
\end{eqnarray*}
\end{PROP}
\begin{proof}
Let us differentiate under the integral in \eqref{DENS_INV_SUB} to get for $u>0$
:
\begin{equation}
\label{PREAL_Z_DENS}
p_{Z^\beta}(t,u)=-\int_0^t \partial_u p_{S^{\beta,+}}(u,y)dy=-\int_0^t L^{*,\beta} p_{S^{\beta,+}}(u,y)dy=-L^{*,\beta} \int_0^t  p_{S^{\beta,+}}(u,y)dy, 
\end{equation}  
using the Kolmogorov forward equation for the density of the subordinator and denoting by $L^{*,\beta} $ the adjoint of its generator. Precisely,
\begin{equation*}
L^{*,\beta}\varphi(t)=-\frac{1}{\Gamma(-\beta)}\int_0^{+\infty}\{\varphi(t-s)-\varphi(t)\} \frac{ds}{s^{1+\beta}}.
\end{equation*}
Recall now that for a smooth (say $C^1$) function $\varphi:\R^+\rightarrow\R $ that is extended by 0 on $\R^- $, for all $t>0$,
\begin{eqnarray*}
-L^{*,\beta}\varphi(t)&=&\frac{1}{\Gamma(-\beta)}\left\{ \int_0^t (\varphi(t-s)-\varphi(t))\frac{ds}{s^{1+\beta}}-\frac{\varphi(t)}{\beta t^\beta}\right\}\\
&=&\frac{1}{\Gamma(-\beta)}\left\{\frac{\varphi(t)-\varphi(0)}{\beta t^\beta}-\int_0^t \varphi'(t-s) s^{-\beta}\frac{ds}{\beta}
-\frac{\varphi(t)}{\beta t^\beta}\right\}\\
&=&\frac{1}{\Gamma(-\beta)(-\beta)}\left\{\int_0^t \varphi'(t-s) s^{-\beta}ds+\frac{\varphi(0)}{t^\beta} \right\}=\D_t^\beta \varphi(t).
\end{eqnarray*}
Thus we get the result for $u>0$ from \eqref{PREAL_Z_DENS} and \eqref{DENS_INV_SUB} differentiating w.r.t. $u$. The proof can be achieved observing that $p_{S^{\beta,+}}(u,y)dy \underset{u\rightarrow 0}{\rightarrow } \delta(dy) $ which implies
$\lim_{u\rightarrow 0^+}p_{Z^\beta }(t,u)= \D_t^\beta \theta(t)=\frac{t^{-\beta}}{\Gamma(1-\beta)}, \ \theta(t)=\I_{t>0} $. 
This observation implies the form of the singularity at the origin in Equation \eqref{EQ_DENS_Z}. 
\end{proof}

We refer to \cite{podl:99} for additional general references concerning fractional calculus.
From \eqref{DENS_GLOB} and \eqref{EQ_DENS_Z} we therefore derive that for all $t>0 $:
\begin{eqnarray}
\label{RIEMMAN_LIOUVILLE}
\D_t^\beta q(t,x)&=&\int_0^{+\infty}p_{S^\alpha}(u,x) \{-\partial_u p_{Z^\beta}(t,u)+\frac{t^{-\beta}}{\Gamma(1-\beta)}\delta(u)\}  du\nonumber\\
&=& \int_0^{+\infty}\partial_u p_{S^\alpha}(u,x) p_{Z^\beta}(t,u)du+\frac{t^{-\beta}}{\Gamma(1-\beta)}\delta(x)\nonumber\\
&=&Lq(t,x)+\frac{t^{-\beta}}{\Gamma(1-\beta)}\delta(x),
\end{eqnarray}
where we used that $\partial_u p_{S^\alpha}(u,x)=L p_{S^\alpha}(u,x)$ (the spatial density satisfies the Kolmogorov backward equation) and the fact that $S_u^\alpha\underset{u\rightarrow 0}{\rightarrow} 0 $. We have also denoted by $\delta $, with a slight abuse of notation, the Dirac mass at 0 in $\R^d$.

In other words, the transition p.d.f. $q$ of $(S_{Z_t^\beta}^\alpha)_{t\ge 0} $ can be viewed as the fundamental solution of \eqref{RIEMMAN_LIOUVILLE}. 
Hence, for a given continuous initial data $f $, the associated Cauchy problem solved by $u(t,x)=\E[f(x+S_{Z_t^\beta}^{\alpha})]$ writes:
\begin{equation}
\label{CAUCHY}
\begin{cases}
\D_t^\beta u(t,x)=Lu(t,x)+\frac{t^{-\beta}}{\Gamma(1-\beta)} f(x),\ (t,x)\in \R+^{*}\times \R^d,\\
u(0,x)=f(x),\ x\in \R^d.
\end{cases}
\end{equation}
Alternatively,  
equation \eqref{CAUCHY} can be rewritten
in terms of the Caputo-Dzherbashyan fractional derivative, denoted $\partial_t^\beta $ hereafter, recalling that for a smooth function $h$ and $\beta \in (0,1) $:
\begin{equation}
\label{CAPUTO}
\begin{split}
\partial_t^\beta h(t):=\frac{1}{\Gamma(1-\beta)}\int_0^t h'(t-u) u^{-\beta}du,\\
\D_t^\beta h(t)-\frac{t^{-\beta}}{\Gamma(1-\beta)}h(0)=\partial_t^\beta h(t),\ t>0.
\end{split}
\end{equation}
The probabilistic representation of the solution to equation \eqref{CAUCHY_I} is thus derived from \eqref{CAUCHY} applying identity \eqref{CAPUTO} to the function $h(t)=u(t,x) $.

Let us mention that, in the special case $L=\lambda^2\Delta$, some explicit expansions of the density of $\sqrt2 \lambda S_{Z_t^\beta}^{2} $ have been obtained by Beghin and Orsingher \cite{begh:orsi:09}. 

The correspondence between the previously described CTRWs and Equation \eqref{RIEMMAN_LIOUVILLE} suggested various extensions. The first one consists in introducing a Markov Chain for the spatial motion, i.e. the $i^ {{\rm th}} $-spatial jump depends on the position $Y_{i-1} $.
 This procedure can lead to consider at the limit differential operators of the more general form
\begin{eqnarray}
\label{SPATIAL_OP}
L_1\phi(x)&=&\langle b(x),\nabla \phi(x)\rangle+\frac 12 \Tr(\sigma\sigma^*(x)D_x^2\phi(x))\ {\text{or}}\nonumber\\
L_2\phi(x)&=&\langle b(x),\nabla\phi(x)\rangle\nonumber\\
&&+ \int_{\R^d} \{\phi(x+\sigma(x)y )-\phi(x)-\nabla \phi(x)\cdot \sigma(x) y \I_{|\sigma(x)y|\le 1} \}\mu(dy),\nonumber\\
\end{eqnarray}
for $\mu(dy):=|y|^{-(1+\alpha)} d|y|\nu (d\bar y) $ as in \eqref{STABLE_GEN},
in equations \eqref{CAUCHY}, \eqref{CAUCHY_I}. We consider homogeneous operators  in time for notational simplicity. 
Let $(X_s)_{s\ge 0} $ solve the SDE with generator given in \eqref{SPATIAL_OP}, i.e.
\begin{eqnarray}
\label{SDE}
X_t&=&x+\int_0^t b(X_s)ds +\int_0^t \sigma(X_{s^-}) dY_s,
\end{eqnarray}
where $Y$ is a Brownian motion in the diffusive case and a symmetric $\alpha $-stable with $\alpha\in(0,2) $  process with L\'evy measure $\mu $ otherwise. Then, under suitable assumptions, needed to get smoothness of the p.d.f. for the spatial motion $(X_t)_{t>0} $, we can derive similarly to the previous computations that the solution to the Cauchy problem \eqref{CAUCHY_I} writes:
\begin{equation}
\label{REP_FK}
\forall (t,x)\in \R^+\times \R^d,\ u(t,x)=\E[f(X_{Z_t^\beta}^{0,x})].
\end{equation}
Observe that, similarly to equation \eqref{DENS_GLOB}, denoting by $p_{X_{Z_T^\beta}} $ the density of ${X_{Z_T^\beta}}$ at time $T>0$, one has for all $(x,y)\in (\R^d)^2$:
\begin{eqnarray}
\label{DENS_GLOB_MARKOV}  
p_{X_{Z_T^\beta}}(x,y)=\int_0^{+\infty}p(u,x,y) p_{Z^\beta}(T,u)  du,
\end{eqnarray}
where $p(u,x,.) $ stands for the density of $X_u^{0,x}$.
Let us also mention that modifying the laws of the waiting  times including a possible dependence on the previous jump time and spatial position, leads to consider at the limit a \textit{modified} fractional derivative in time. We refer to Kolokoltsov \cite{kolo:09} for details and results concerning the extension of the formula \eqref{REP_FK} and convergence of the associated approximations.

\section{Assumptions and Main Results}
\label{MAIN_RESULTS}
\subsection{Assumptions on the Coefficients and Approximation Scheme.}
We assume that the coefficients in 
equation \eqref{SDE} 
satisfy the following conditions:
\begin{trivlist}
\item[\A{S}] The coefficients $b:\R^d\rightarrow \R^d,\ \sigma:\R^d\rightarrow \R^d\otimes \R^d$ are assumed to be
bounded as well as their derivatives up to order 6.
\item[\A{UE}] There exists $\Lambda\ge 1$ s.t. for all $(x,\xi)\in (\R^d)^2 $:
$$\Lambda ^{-1}|\xi|^2 \le \langle \sigma\sigma^*(x)\xi,\xi\rangle \le \Lambda|\xi|^2.$$
Now,   we consider for a given time step $h>0$, setting for $i\in \N,\ t_i:=ih$, the following Markov Chain approximation of equation \eqref{SDE}:
\begin{equation}
\label{SCHEME}
\forall i\in \N,\ X_{t_{i+1}}^h=X_{t_i}^h+b(X_{t_i}^h)h+\sigma(X_{t_i}^h)h^{1/\alpha} \eta_{i+1},\ X_0^h=x,
\end{equation}
for i.i.d. $\R^d$-valued random variables $(\eta_j)_{j\ge 1} $ and $\alpha\in (0,2]$. 

Also, since we are considering the fractional derivative, we can explicitly and exactly simulate $(S_{t_i}^{\beta,+})_{i \in \N}$, see, e.g., Weron and Weron \cite{wero:wero:95}. From the exact simulation on the time grid, a natural choice to approximate the inverse process $Z_t^\beta$ is to consider the \textit{discrete} inverse process:
\begin{equation}
\label{OBS_DIS_INV}
\forall t\ge 0,\ Z_t^{\beta,h}:=\inf\{s_i:=ih: S_{s_i }^{\beta,+}> t\}.
\end{equation}

Thus, for a given $T>0$, we finally approximate $X_{Z_T^\beta} $ appearing in \eqref{REP_FK} by $X_{Z_T^{\beta,h}}^h $.

Let us mention that the smoothness assumption \A{S} could be weakened in the case of the Euler scheme. These smoothness conditions are actually those required to apply the results of \cite{kona:mamm:09}.

\subsection{Diffusive case:} This case is associated with spatial motions with generators of the form $L_1$ in \eqref{SPATIAL_OP}.
We then specifically assume that $\alpha=2 $  in \eqref{SCHEME} and that the $(\eta_j)_{j\ge 1} $ are 
 s.t.  their moments coincide with those of the standard Gaussian law in $\R^d $ up to order $2$. 
We consider two kinds of conditions: 
\begin{trivlist}
\item[\A{I${}_{{\rm Eul}} $}] The $(\eta_j)_{j\ge 1} $ have standard Gaussian densities (so that the above scheme is the usual Euler discretization).
\item[\A{I${}_m$}]  For a given integer $m\ge 2(d+1)$, the $(\eta_j)_{j\ge 1} $ have a density $Q_M $ which is $C^4$ and has, together with its derivatives up to order $4$, polynomial decay of order $M> d(2m+1)+4$. Namely, for all $z\in \R^d$ and multi-index $\nu, |\nu|\le 4 $:
$$|D_\nu Q_M(z)|\le C_M(1+|z|)^{-M}.$$
\end{trivlist}
\end{trivlist}
We say that assumption \A{A${}_{D,{\rm Eul}}$} (resp. \A{A${}_{D,m}$}) is in force as soon as \A{S}, \A{UE}, \A{I${}_{{\rm Eul}}$} (resp. \A{I${}_{m}$}) hold.
We write \A{A${}_{D}$} whenever \A{A${}_{D,{\rm Eul}}$} or \A{A${}_{D,m}$} holds. 

It is well known that under \A{A$_{D} $} the random variables $X_t,X_{t_i}^h $ have densities for all $t>0, i\in \N^* $ respectively, see e.g. \cite{kona:mamm:00}.
This property then transfers to $X_{Z_T^{\beta}}$ and $X_{Z_T^{\beta,h}}^h$ thanks to the convolution equation \eqref{DENS_GLOB_MARKOV} and its discrete analogue for the approximation (see equation \eqref{DEC_1_SCHEME} below). We will denote  by $p_{X_{Z_T^\beta}} $ and $p_{X_{Z_T^{\beta,h}}^h} $ the associated p.d.f. To distinguish precisely the approximations under \A{A${}_{D,{\rm Eul}}$} or \A{A${}_{D,m}$} we will specifically denote, when needed, by $X_{Z_T^{\beta,h}}^{{\rm Eul},h}, p_{X_{Z_T^{\beta,h}}^{{\rm Eul},h}} $ the Euler approximation and its density.

We now have the following convergence result:
\begin{THM}[Error Bounds for the Approximation Schemes]
\label{THM_DIFF}
\hspace*{2.cm}
\begin{trivlist}
\item[-] \textbf{Euler Scheme.}
There exists $c:=c($\A{A${}_D $}) s.t. for a given time step $h\in (0,1)$, a deterministic time horizon $T>0$ s.t. $T>h^{1/2}$ (not necessarily a multiple of $h$), for every $x\neq y$ we have under \A{A$_{D,{\rm Eul}}$} (Euler scheme):
\begin{eqnarray}
|p_{X_{Z_T^\beta}}(x,y)-p_{X_{Z_{T}^{\beta,h}}^{{\rm Eul},h} }(x,y)|\nonumber\\
=
|\int_{\R^+}p_{Z^\beta}(T,u) p(u,x,y) du- \sum_{i\ge 1} \P[Z_T^{\beta,h}=t_i]  p^{{\rm Eul},h}(t_i,x,y)|\label{DEC_1_SCHEME}\\
\le 
ch \bigg( {\mathcal E}_{\beta,{\rm Time}}(T,x-y)+{\mathcal E}_{\beta, {\rm Space}}(T,x-y) \bigg),\nonumber
\end{eqnarray}
where:
\begin{eqnarray*}
{\mathcal E}_{\beta,{\rm Time}}(T,x-y)&:=&\bigg\{ \bigg(\frac{\I_{d\le 2}}{T^{\beta/2}|x-y|} +\frac{\I_{d\ge 3}}{|x-y|^2}\bigg)\hat p_\beta
(T,x-y) +\frac1{T^\beta}\tilde p_{\beta}(T,x-y)\bigg\},\nonumber\\
{\mathcal E}_{\beta, {\rm Space}}(T,x-y)&:=&\bigg\{ \bigg(\frac{\I_{d=1}+\I_{d\ge 3}}{|x-y|} +\frac{\I_{d= 2}}{T^{\beta/2}}\bigg)\hat p_\beta
(T,x-y) +\frac1{T^{\beta/2}}\tilde p_{\beta}(T,x-y)\bigg\},\nonumber
\end{eqnarray*}
with 
\begin{eqnarray}
\label{DEF_P_BETA}
\hat p_{\beta}
(T,x-y)&:=&\bigg(\frac{\I_{d\le 2}}{T^{\beta/2}|x-y|^{\I_{d=2}}}+
\frac{\I_{d\ge 3}}{T^{\beta}|x-y|^{d-2}} \bigg)\exp(cT^{\beta/2})\nonumber\\
&&\times\exp{\left(-c^{-1}\frac{|x-y|^2}{T^\beta}\right)},\nonumber\\
\tilde p_{\beta}(T,x-y)&:=& \frac{\exp(cT^{\beta/(1+\beta)})}{T^{\beta d/2}} \exp\bigg(-c^{-1}\left\{\frac{|x-y|^2}{T^\beta} \right\}^{1/(2-\beta)}\bigg)\bigg\}. 
\end{eqnarray}
The two auxiliary functions $\hat p,\tilde p_\beta$ are s.t. their sum serves as an upper bound for both $p_{X_{Z_T^\beta}},\ p_{X_{Z_{T}^{\beta,h}}^{{\rm Eul},h} }$. Precisely,
there exists $c:=c(\beta) $ s.t. for all $T>0,\ (x,y)\in (\R^d)^2, \ x\neq y $:
$$(p_{X_{Z_T^\beta}}+ p_{X_{Z_{T}^{\beta,h}}^{{\rm Eul},h} })(x,y)\le c(\hat p_\beta
+\tilde p_\beta)(T,x-y).$$
We refer to Corollary \ref{LEMME_QUI_FAIT_A_LA_KOCHUBEI} for a proof of this last statement. We also establish in Appendix \ref{APP_HK} a similar lower bound, proving the estimate is sharp.\newpage

\item[-] \textbf{Markov Chain.}
Consider now assumption \A{A${}_{D,m}$} (general Markov Chain). Then,  there exists $c:=c($\A{A${}_{D,m}$})$ $ and for all $\varepsilon \in (0,1/5) $, $c_\varepsilon:=c_\varepsilon($\A{A${}_{D,m}$}) s.t. for all given time step $h\in (0,1)$, $T^\beta\ge h^{1/5-\varepsilon},\ (x,y)\in (\R^d)^2,\ x\neq y $:
\begin{equation}
\begin{split}
|p_{X_{Z_T^\beta}}(x,y)-p_{X_{Z_{T}^{\beta,h} }^h }(x,y)|
\le c \big\{h {\mathcal E}_{\beta,{\rm Time}}(T,x-y)\\+c_\varepsilon\{
h^{1/2} {\mathcal E}_{\beta,{\rm Space, LLT}}^M(T,x-y)
+
{\mathcal E}_{\beta,{\rm Space, NoLLT}}^{M,
\varepsilon}(T,x-y,h)\}
\big\},
\end{split}
\label{decoup_err_mc}
\end{equation}
where the contribution ${\mathcal E}_{\beta,{\rm Time}}(T,x-y) $ due to the time sensitivity, see equations \eqref{DEC_1_SCHEME} and \eqref{ERROR_1}, is as above and:
\begin{eqnarray*}
{\mathcal E}_{\beta,{\rm Space, LLT}}^M(T,x-y)&:=&\bigg\{ \bigg(\frac{\I_{d= 1}+\I_{d\ge 3}}{|x-y|}+\frac{\I_{d= 2}}{T^{\beta/2}} \bigg)\hat q_{m-[(d-1)+\I_{d=1}],\beta}(T,x-y)\\ && +\frac1{T^{\beta/2}}\tilde q_{m,\beta}(T,x-y)\bigg\},
\end{eqnarray*}
denoting as well for all $l\in \N^*  $: 
\begin{eqnarray}
\label{DEF_Q_M_BETA}
\hat q_{l,\beta}(T,x-y)&:=&\bigg(\frac{\I_{d\le 2}}{T^{\beta/2}|x-y|^{\I_{d=2}}}+
\frac{\I_{d\ge 3}}{T^{\beta}|x-y|^{d-2}} \bigg)\exp(cT^{\beta/2})\nonumber\\
&&\times\left(1+\frac{|x-y|}{T^{\beta/2}}\right)^{-l},\nonumber\\
\tilde q_{m,\beta}(T,x-y)&:=& \frac{\exp(cT^{\beta/(1+\beta)})}{T^{\beta d/2}} 
\bigg(1+\frac{|x-y|}{T^{\beta/2}}\bigg)^{-\lfloor \frac{m}{2-\beta} \rfloor}.
\end{eqnarray}
Also:
\begin{eqnarray*}
{\mathcal E}_{\beta,{\rm Space, NoLLT}}^{M,
\varepsilon}(T,x-y,h):=
 \frac{(h^{1/5-\varepsilon})^{\frac 12\I_{d\le2}}}{T^\beta|x-y|^{d-2+\I_{d\le2}}} \exp(cT^{\beta/2})\\
\times\bigg[\exp\left(-c\frac{|x-y|^2}{h^{1/5-\varepsilon}}  \right) +\left(1+\frac {|x-y|}{(h^{1/5-\varepsilon})^{1/2}} \right)^{-2(m-1)+(d-2)+\I_{d\le2}} \bigg].
\end{eqnarray*}

\end{trivlist}

\end{THM}

\textbf{Comments on the results.}
\begin{trivlist}
\item[-] \textbf{Euler Scheme.}
Observe that the term $\hat p_\beta$ in \eqref{DEF_P_BETA} comes from the contribution of small times in fomula \eqref{DENS_GLOB_MARKOV}. 
This gives the additional (diagonal) spatial singularity in the fractional case. The term $\tilde p_\beta$ in \eqref{DEF_P_BETA} comes from the integration over long times in \eqref{DENS_GLOB_MARKOV}. It induces a loss of concentration in space w.r.t to the usual case $\beta=1$, though preserving the usual parabolic equilibrium: the spatial contribution remains comparable to the square root of time.

\begin{trivlist}
\item[$\bullet$]In the diagonal region, i.e. for points s.t. $  |x-y| \le  \kappa T^{\beta/2}$ for some $\kappa \ge 1 $, if we restrict to spatial points that are equivalent to the time contribution for the usual parabolic metric, i.e.  $\kappa^{-1} T^{\beta/2} \le   |x-y| \le  \kappa T^{\beta/2}$, the previous error bound reads in \textit{small} time, i.e. for $T\le 1$ : 
\begin{equation}
\label{BD_EQUIV}
|p_{X_{Z_T^\beta}}(x,y)-p_{X_{Z_{T}^{\beta,h}}^{{\rm Eul},h} }(x,y)|\le \frac{ch}{T^\beta}(\hat p_\beta+\tilde p_\beta)(T,x-y).
\end{equation}
This means that we find the upper-bound for the density of the processes involved, multiplied by a factor  corresponding to the singularity induced by a time derivative (or equivalently second order derivatives in space) of the non degenerate diffusive heat kernel at time $T^\beta$. There is even here a difference w.r.t. to the usual analysis, see e.g. \cite{kona:mamm:02}, which involves a singularity of order one in space, which in the current bound comes from the contribution ${\mathcal E}_{\beta,{\rm Space}}(T,x-y) $ in \eqref{DEC_1_SCHEME} and is negligible if $T\le 1$ . This is due to the discrete approximation of the inverse of the subordinator which yields to consider derivatives in time for the densities of the spatial motion (see equation \eqref{ERROR_1} and Lemma \ref{LEMME_T_SENS} below). If the error bound \eqref{DEC_1_SCHEME} is now considered for $T$ sufficiently large, the contribution 
${\mathcal E}_{\beta,{\rm Space}}(T,x-y) $ dominates and we are back to the \textit{usual} error rate for the considered region.

Observe as well that the spatial diagonal singularity induced by the fractional time derivative yields additional constraints. Namely, 
for the Euler approach, to get a convergence rate of order $h^\varepsilon $ one needs to take:
$|x-y|\ge h^{(1-\varepsilon)/d}$, $d\ge 1 $.

\item[$\bullet $] In the off-diagonal region, i.e. $|x-y|\ge \kappa T^{\beta/2} $, we do not have spatial singularities anymore. Indeed, the error bound also reads as in \eqref{BD_EQUIV}. Anyhow, the specificity comes from the loss of concentration due to the term $\tilde p_\beta $ coming from the long time integration in \eqref{DENS_GLOB_MARKOV} and appearing in the two-sided estimates of Appendix \ref{APP_HK}. We refer to Lemmas \ref{BOUNDS_P_Z_P_Z_H} and \ref{GROS_LEMME} (see also Corollary \ref{LEMME_QUI_FAIT_A_LA_KOCHUBEI} and Theorem \ref{HK_BOUNDS_FRAC}) for a detailed presentation of these facts.
\end{trivlist}

The previous facts make us feel that the previous bounds are sharp, up to constants.

\item[-] {\textbf{Markov Chains.}} For the error bound \eqref{decoup_err_mc} associated with the Markov Chain approximation, the term ${\mathcal E}_{\beta,{\rm Time}} $ 
is similar to the Euler case, since it does not depend on the chosen discretization scheme but only on the time sensitivities of the spatial density, see again \eqref{ERROR_1}. On the other hand two additional contributions appear. The first one, ${\mathcal E}_{\beta,{\rm Space, LLT}}^M $ enjoys the same convergence rate $h^{1/2} $ as in the usual LLT (see Petrov \cite{petr:05}, \cite{kona:mamm:00}) and the terms $ \hat q_{m-(d-1)+\I_{d=1},\beta}, \tilde q_{m,\beta}$ appearing in \eqref{decoup_err_mc} are the \textit{Markov chain analogue} of the contributions $\hat p_\beta $ and $\tilde p_\beta$ in \eqref{DEC_1_SCHEME}\footnote{We recall that $(\hat p_\beta+\tilde p_\beta)(T,x-y) $ serves as an upper bound for $(p_{X_{Z_T^\beta}})(x,y) $ (see Corollary \ref{LEMME_QUI_FAIT_A_LA_KOCHUBEI}).}.
Note first that we keep under \A{A${}_{D,m}$} the same index $m$ as the one appearing for the weak error at a given fixed time, see \eqref{LLT_DIFF_MARG} below and \cite{kona:mamm:09}, \cite{kona:mamm:00}, for the contribution $\tilde q_{m,\beta} $ associated with the large time integration and which also yields a loss of concentration. Observe as well that the spatial diagonal singularities in small time also deteriorate the concentration in this framework. 
We point out that, for the LLT to apply, the considered time has to be sufficiently large, namely $u\ge h^{1/5-\varepsilon} $ in \eqref{DENS_GLOB_MARKOV}. The last term ${\mathcal E}_{\beta,{\rm Space, NoLLT}}^{M,
\varepsilon}$ comes precisely from those \textit{very} small times $u\le h^{1/5-\varepsilon} $ for which the limit results do not apply.
\item[-] {\textbf{Continuity of the Estimates when $\beta \uparrow 1$.}} 
In this case, the Caputo derivative in \eqref{CAPUTO} tends to the usual derivative. It is therefore natural to ask whether our results match the ones previously obtained in that case (see \cite{kona:mamm:00}, \cite{kona:mamm:02}, \cite{kona:meno:10}).
Let us now emphasize that when $\beta\uparrow 1 $ then $\psi_{S_u^{\beta,+}}(\lambda)\rightarrow \exp(-\lambda u), \lambda \ge 0$. Thus, the subordinator $S^{\beta,+} $, and its inverse $Z^\beta $ both tend to the deterministic drift with slope 1. This means that for a given $T>0, \P[Z_T^\beta \in du]\underset{\beta \rightarrow 1}{\longrightarrow} \delta_T(du)$. Hence, the time integral in \eqref{DENS_GLOB_MARKOV} disappears at the limit, i.e. $ p_{X_{Z_T^\beta}}(x,y)\underset{\beta \rightarrow 1}{\longrightarrow}p(T,x,y)$.
From \eqref{DEC_1_SCHEME} and \eqref{ERROR_1}, the same phenomenon occurs for the scheme provided $T=Kh,\ K\in \N^*$ and the error associated with the time sensitivity, term ${\mathcal E}_1 $ in \eqref{ERROR_1} yielding the contribution ${\mathcal E}_{\beta,{\rm Time}} $ would vanish in the previous error bounds. Since the previous Dirac convergence would also kill the additional contributions in small and long time, we finally derive that
the error bounds obtained in \eqref{DEC_1_SCHEME} and \eqref{decoup_err_mc} are coherent with those of the previous works at the limit.
\end{trivlist}

\newpage
\subsection{Strictly Stable case}
Additionally to assumptions \A{A}, \A{UE} we will assume that :
\begin{trivlist}
\item[\A{ND}]  The spherical part $\nu $ of the measure
$\mu$ in \eqref{SPATIAL_OP} satisfies the non-degeneracy condition of equation
\eqref{ND_SPEC_MEAS}.
Moreover,
$\nu $ has a $C^3$ density w.r.t. the Lebesgue measure of the sphere.  
\item[\A{B}] The drift $b=0$ if $\alpha\le 1 $.
\end{trivlist}
This last condition is rather usual to have pointwise bounds on the density of the SDE, see \cite{kolo:00}.
We will consider similarly to the diffusive case the approximation \eqref{SCHEME} for $\alpha\in (0,2) $ restricting ourselves to the Euler scheme case, i.e.
the $(\eta_j)_{j\ge 1} $ are i.i.d. symmetric stable random variables with common law $Y_1 $ (driving process at time $1$). We say that \A{A${}_S $} holds if \A{S}, \A{UE}, \A{ND}, \A{B} are in force. 
Recall from Kolokoltsov \cite{kolo:00} and \cite{kona:meno:10} that the random variables $X_t,X_{t_i}^{h,{\rm Eul}} $ have a density for all $t>0,\ i\in \N^* $ respectively. Thus, similarly to the diffusive case, this property transfers to the random variables $X_{Z_T^\beta}, X_{Z_T^{\beta,h}}^{h,{\rm Eul}} $ from \eqref{DENS_GLOB_MARKOV}.
With the previous notations, we have the following approximation result.

\begin{THM}
\label{THM_STABLE} 
Under \A{A$_{S}$}, there exists a constant $c$ s.t. for a given time step $h\in (0,1)$ and a fixed time horizon $T>h^{1/\beta}$,  and any $(x,y)\in ( \R^ d)^2$ s.t. $|x-y|\ge h^{1/\alpha} $  we have the following convergence result:  
\begin{equation}
\label{EQ_ERR_STABLE_EUL}
\begin{split}
|p_{X_{Z_T^\beta}}(x,y)-p_{X_{Z_{T}^{\beta,h} }^{h,{\rm Eul}} }(x,y)|\\
\le  
 c h\Big \{\Big( \frac{ \I_{|x-y|\le T^{\beta/\alpha}} }{|x-y|^\alpha}+\frac{ \I_{|x-y|> T^{\beta/\alpha}} }{T^\beta}\Big)\hat p_{\beta}(T,x-y)+\frac{1}{T^\beta}\tilde p_\beta(T,x-y)\Big\}
 , 
 \end{split}
\end{equation}
where
\begin{eqnarray*}
 \hat p_{\beta}(T,x-y) &:= &\exp(cT^{\beta\omega)}) \Big\{ \frac{1}{T^\beta|x-y|^{d-\alpha}}\I_{|x-y|\le T^{\beta/\alpha} }+\frac{T^\beta}{|x-y|^{d+\alpha}}\I_{|x-y|>T^{\beta/\alpha}}\Big\},\\
 \tilde p_\beta(T,x-y)&:=&\frac{\exp(cT^{\beta\omega/(1-\omega(1-\beta))})}{T^{\beta d/\alpha}\left(1+\frac{|x-y|}{T^{\beta/\alpha}} \right)^{d+\alpha}},\ \omega:=\frac 1\alpha \wedge 1,\\
(p_{X_{Z_T^\beta}}&+&p_{X_{Z_T^{\beta,h}}^{{\rm Eul},h} })(x,y)\le c (\hat p_{\beta}+\tilde p_\beta)(T,x-y),\ c:=c(\beta,\A{A_{\text{\it S}}}).
\end{eqnarray*}
\end{THM}


\begin{REM}
We observe a phenomenon which is similar to the diffusive case, i.e. the error bound has the expected rate of order $h$ up to additional singularities. In the strictly stable case, i.e. $\alpha \in (0,2) $, two contributions give additional singularities w.r.t. those already appearing in the sharp density bounds (see Theorem \ref{HK_BOUNDS_FRAC}). Those contributions write $T^{-\beta}, |x-y|^{-\alpha} $ and  come from the time derivative of the density in \eqref{DENS_GLOB_MARKOV}.  

 
On the other hand, since the spatial motion already has heavy tails, we do not observe here a loss of concentration phenomenon as we did in the diffusive case.
Eventually, we still have diagonal spatial singularities, which are in the strictly stable case more direct to formulate. They appear precisely when $|x-y|\le T^{\beta/\alpha} $, that is when the diagonal regime holds w.r.t. to the usual stable parabolic scaling. This is again the effect of the fractional in time derivative analyzed in small time, see \eqref{DENS_GLOB_MARKOV}.
Let us mention that there is some continuity w.r.t. to the stability index for the induced singularity, at least when $d\ge 3$ since the diffusive case provides a factor  $(T^\beta|x-y|^{d-2})^{-1} $ for that contribution in $\hat p_\beta(T,x-y) $.
\end{REM}

\section{Proof of  the Main Results.}
  \label{PROOF}
In the following we denote by $c$ a generic positive constant that may change from line to line. Explicit dependencies for those constants $c$, mainly on assumptions \A{A${}_D$} or 
\A{A${}_S$}, $\beta $, are specified as well.

\subsection{Decomposition of the Error.}
For given points  $(T,x,y)\in \R_+^{*}\times (\R^d)^2 $, we get from \eqref{DENS_GLOB_MARKOV} and \eqref{OBS_DIS_INV}
that the error writes:
\begin{eqnarray}
\label{ERROR_1}
{\mathcal E}(T,x,y)&:=&(p_{X_{Z_T^\beta}}-p_{X_{Z_T^{\beta,h}}^h })(x,y)\nonumber\\
&=&\int_{\R^+}p_{Z^\beta}(T,u) p(u,x,y) du- \sum_{i\ge 1} \P[Z_T^{\beta,h}=t_i]  p^h(t_i,x,y),\nonumber \\
&=&\int_{\R^+}p_{Z^\beta}(T,u) (p(u,x,y)-p(\phi(u),x,y)) du\nonumber\\
&& + \{\sum_{i\ge 1} \P[Z_T^{\beta,h}=t_i] (p(t_i,x,y)- p^h(t_i,x,y))\},\nonumber\\
&:=&{\mathcal E}_{1}(T,x,y)+{\mathcal E}_{2}(T,x,y),
\end{eqnarray}  
denoting for all $u\in \R_+^{*},\ \phi(u):=\inf\{t_i:=ih: t_i> u\}$, where we have used \eqref{OBS_DIS_INV} for the last but one equality. Let us specify that 
${\mathcal E}_1 $ is a term involving the time sensitivities of the density of the SDE \eqref{SDE}, which does not depend on the approximation scheme.
The contribution ${\mathcal E}_2 $ is associated with the spatial sensitivity involving the discretization error for the considered scheme.

 The idea is then to use some known asymptotics of the density $S^{\beta,+}$ (see e.g. Zolotarev \cite{zolo:92} or Hahn \textit{et al.} \cite{hahn:koba:umar:11}) which provide those for $p_{Z^\beta} $. On the other hand, the quantities 
 $|p(t_i,x,y)-p^h(t_i,x,y)|$ have been thoroughly investigated in the literature for the diffusive case, see \cite{kona:mamm:00}, \cite{gobe:laba:08}. 
 The results established therein give that for a general approximation scheme of type \eqref{SCHEME} the difference is controlled at order $h^{1/2}$ which is the convergence rate in the Gaussian LLT, see e.g. Petrov \cite{petr:05}. If the scheme already involves Gaussian increments (usual Euler scheme of the diffusion) the control is improved and the convergence rate is $h$, see also \cite{kona:mamm:02}.
 In the strictly stable case, i.e. $\alpha<2 $,
 the convergence rate of the Euler scheme, for which the exact stable increments are considered in \eqref{SCHEME}, has been investigated in \cite{kona:meno:10}. The proof of Theorems \ref{THM_DIFF} and \ref{THM_STABLE} immediately follows from equation \eqref{ERROR_1} and the following Lemmas.
 
 \begin{LEMME}[Time Sensitivity in the Error]
\label{LEMME_T_SENS}
 Under assumption \A{A${}_D $} (Diffusive case) we have that there exists $c:=c(\beta, $\A{A${}_D$}) $\ge 1$,
  s.t. for all $T>0, (x,y)\in (\R^d)^2,\ x \neq y$:
 \begin{equation*}
\begin{split}
|{\mathcal E}_{1}(T,x,y)|\le ch \bigg\{
\frac{\exp(cT^{\beta/2})}{T^{\beta} |x-y|^{d}}\exp(-\frac{c^{-1}}2\frac{|x-y|^2}{T^\beta})\\
+ \frac{\exp(cT^{\beta/(1+\beta)})}{T^{\beta(1+d/2)}} \exp\bigg(-c^{-1}\left\{\frac{|x-y|^2}{T^\beta} \right\}^{1/(2-\beta)}\bigg)\bigg\}\\
\le ch\bigg\{ \bigg(\frac{\I_{d\le 2}}{T^{\beta/2}|x-y|} +\frac{\I_{d\ge 3}}{|x-y|^2}\bigg)\hat p_\beta(T,x-y) +\frac1{T^\beta}\tilde p_{\beta}(T,x-y)\bigg\}.
\end{split}
 \end{equation*}
 
 Under Assumption \A{A${}_S $} (Strictly stable case), there exists $c:=c(\beta,\A{A_S})
 $ s.t. for all $T>0, (x,y)\in (\R^d)^2,\ x \neq y $:
 \begin{equation*}
\begin{split}
|{\mathcal E}_{1}(T,x,y)| 
&\le c h\bigg \{\exp(cT^{\beta\omega})\bigg(\frac{\I_{|x-y|\le T^{\beta/\alpha}}}{T^\beta |x-y|^{d}}+\frac{\I_{|x-y|>T^{\beta/\alpha} }}{|x-y|^{d+\alpha}}\bigg)+\frac{\exp(cT^{\beta\omega/(1-\omega(1-\beta))})}{T^{\beta(1+d/\alpha)}\vee |x-y|^{d+\alpha}}\bigg\}\\
&\le c h\bigg\{ \bigg(\frac{\I_{|x-y|\le T^{\beta/\alpha}}}{|x-y|^\alpha}+\frac{\I_{|x-y|>T^{\beta/\alpha}}}{T^\beta}\bigg)\hat p_\beta(T,x-y)+\frac{1}{T^\beta}\tilde p_\beta(T,x-y) \bigg\},
\end{split}
\end{equation*}
 where $\omega=\frac 1\alpha\wedge 1 $.
 \end{LEMME} 
 \begin{REM}
 Observe that in the stable case there is no loss of concentration as in the diffusive one. The second term of the stable bound in Lemma \ref{LEMME_T_SENS} corresponds to the \textit{usual} estimate for the time derivative of the stable density, see Lemma \ref{HKUSUEL} below and \cite{kolo:00}. The first one corresponds to the additional spatial singularity induced by the fractional derivative in time in the diagonal regime $|x-y|\le T^{\beta/\alpha} $ (stable parabolic scaling at time $T^\beta $).
 \end{REM}

 We now give two Lemmas concerning the spatial sensitivity that involve the analysis of the discretization error and the time randomization. \begin{LEMME}[Spatial Sensitivity: Diffusive Case]
 \label{LEMME_S_SENS_D}
 Under assumption \A{A${}_D $}   there exists  $ c:=c(\beta,$ \A{A${}_D$})
  s.t. for all $T>0,\ (x,y)\in (\R^d)^2,\ x\neq y $:
 \begin{trivlist}
\item[-] For the Euler scheme: 
 \begin{equation}
\label{CTR_E12_EUL}
\begin{split}
|{\mathcal E}_{2}^{{\rm Eul}}(T,x,y)|& \le
ch \bigg\{
\bigg\{\frac{\I_{d=1}+\I_{d\ge 3}}{|x-y|}+\frac{\I_{d=2}}{T^{\beta/2}}\bigg\}\hat p_\beta(T,x-y)
+\frac{1}{T^{\beta/2}}\tilde p_\beta(T,x-y)
\bigg\}\\
& := ch {\mathcal E}_{\beta,{\rm Space}}(T,x-y),
\end{split}
\end{equation}
where we write ${\mathcal E}_2^{{\rm Eul}} $, with a slight difference w.r.t. the definition in \eqref{ERROR_1}, to specifically make the distinction with general Markov Chains. 
\item[-] For a general Markov Chain, we have that for all $\varepsilon\in (0,1/5) $ there exists $c_\varepsilon:=c_\varepsilon($\A{A${}_{D,m}$}) s.t. for all 
$T>0$, $(x,y)\in (\R^d)^2,\ x\neq y $: 
\begin{equation}
\label{CTR_E12}
\begin{split}
|{\mathcal E}_{2}(T,x,y)|
\le  c_\varepsilon \bigg[h^{1/2} \bigg\{\bigg\{\frac{\I_{d=1}+\I_{d\ge 3}}{|x-y|}+\frac{\I_{d=2}}{T^{\beta/2}}\bigg\}
\hat q_{m-(d-1)+\I_{d=1},\beta}(T,x-y)\\
+\frac{1}{T^{\beta/2}}\tilde q_{m,\beta}(T,x-y) \bigg\}
+  \frac{(h^{1/5-\varepsilon})^{\frac 12 \I_{d\le2}}
}{T^\beta|x-y|^{d-2+\I_{d\le2}}} \exp(cT^{\beta/2})\\
\times\bigg[\exp\left(-c\frac{|x-y|^2}{h^{1/5-\varepsilon}}  \right) +\left(1+\frac {|x-y|}{(h^{1/5-\varepsilon})^{1/2}} \right)^{-2(m-1)+(d-2)+\I_{d\le2}} \bigg]\bigg]\\
:= c_\varepsilon\{h^{1/2} {\mathcal E}_{\beta,{\rm Space, LLT}}^M(T,x-y)
+{\mathcal E}_{\beta,{\rm Space, NoLLT}}^{M,
\varepsilon}(T,x-y,h) \}.
\end{split}
\end{equation}
\end{trivlist}
\end{LEMME}

\begin{LEMME}[Spatial Sensitivity: Strictly Stable Case]
\label{LEMME_S_SENS_S}
Under assumption \A{A${}_{S} $}, there exists $c:=c(\beta, $ \A{A${}_S$})$\ge 1 $ s.t. for all $T>0 $, $(x,y)\in (\R^d)^2 ,\ x\neq y$:
\begin{eqnarray*}
|{\mathcal E}_{2}^{{\rm Eul}}(T,x,y)|& 
\le c h\bigg \{\exp(cT^{\beta\omega})\bigg(\frac{\I_{|x-y|\le T^{\beta/\alpha}}}{T^\beta |x-y|^{d}}+\frac{\I_{ |x-y|>T^{\beta/\alpha}}}{|x-y|^{d+\alpha}}\bigg)+\frac{\exp(cT^{\beta\omega/(1-\omega(1-\beta))})}{T^{\beta(1+d/\alpha)}\vee |x-y|^{d+\alpha}}\bigg\}\nonumber \\
&\le c h\bigg\{ \bigg(\frac{\I_{|x-y|\le T^{\beta/\alpha}}}{|x-y|^\alpha}+\frac{\I_{|x-y|>T^{\beta/\alpha}}}{T^\beta}\bigg)\hat p_\beta(T,x-y)+\frac{1}{T^\beta}\tilde p_\beta(T,x-y) \bigg\}.
\end{eqnarray*} 

\end{LEMME}  
\begin{REM}
We emphasize that in the strictly stable case, both time and spatial sensitivities yield the same error. This comes from the fact that the short time singularity for the error expansion   
of the Euler scheme, exactly matches the time singularity of the derivative of the stable heat kernel, see equations \eqref{CTR_DIFF_S_TIME} and \eqref{HK_BOUND_STABLE}.  
\end{REM}


\subsection{Asymptotics of the PDF of the Inverse Subordinator.}
As a preliminary remark, let us mention that in the particular case $\beta=2^{-n},\ n\in \N^* $, for all $T>0$, the random variable $Z_T^\beta $, inverse at the level $T$ of the subordinator
$S^{\beta,+} $ can be easily and explicitly simulated from the following Proposition. In particular, the auxiliary discrete approximation 
$Z_T^{\beta,h} $ is not needed in those cases, and the time sensitivity error ${\mathcal E}_1(T,x,y) $ in \eqref{ERROR_1} would vanish (considering a straightforward extension of the discretization scheme for the last time step).

\begin{PROP}[Identity in law for $Z_T^\beta,\ \beta=2^{-n},\ n\in \N^*,\ T>0$]
\label{ID_LOI_PUISSN}
\begin{equation}
\label{ID_LAW}
\forall n\in \N^*,\ \beta=2^{-n}, Z_T^\beta\overset{({\rm law})}{=} \sqrt 2|B^1(\sqrt 2| B^2(\sqrt 2| B^3(\cdots(\sqrt 2 |B^n(T)|)\cdots)|)|)|,
\end{equation}
where the $(B^i)_{i\in \{1,\cdots,n\}}  $ are independent Brownian motions. 
\end{PROP}
\begin{proof}
Let us start with $n=1$. In this case, it is well known that, recalling that $(S_t^{\frac 12,+})_{t\ge 0} $ is the standard subordinator with index $\frac12 $, one has
$\left(S_t^{\frac 12,+}\right)_{t\ge 0}\overset{({\rm law})}=\left(\tau_{\frac t{\sqrt 2}}\right)_{t\ge 0}$, where $\tau_{\frac t {\sqrt 2}}
:=\inf\{s\ge 0: B_s=\frac t{\sqrt 2}\}$, stands for the hitting time of level $\frac t{\sqrt 2} $ for the standard Brownian motion $B$. Indeed, considering for a given $\lambda\ge 0 $ the exponential martingale $\left(\exp(\sqrt{2\lambda}B_{s}-\lambda s) \right)_{s\ge 0}$, the stopping theorem gives that, for all $s\ge 0$,
\begin{eqnarray*}
\E[\exp(\sqrt{2\lambda}B_{s\wedge \tau_{\frac t{\sqrt 2}}}-\lambda (s\wedge \tau_{\frac t{\sqrt 2}}))]=1\underset{s\rightarrow +\infty}{\Longrightarrow} \E[\exp(-\lambda \tau_{\frac t{\sqrt 2}} )]=\exp(- t \lambda^{1/2}).
\end{eqnarray*}
Write now for all $t\ge 0 $:
\begin{eqnarray*}
\P[Z_T^{\frac 12}\le t]=\P[S_t^{\frac 12,+}> T]=\P[\tau_{\frac t{\sqrt 2}}>T]=\P[\sup_{s\in [0,T]} B_s\le \frac{t}{\sqrt 2}]=\P[\sqrt 2|B_T|\le t],
\end{eqnarray*}
using the well known L\'evy identity $\sup_{s\in [0,T]} B_s\overset{({\rm law})}{=} |B_T|,\ T>0 $, for the last equality, see \cite{revu:yor:99}. This proves \eqref{ID_LAW} for $n=1$. Introduce now a family of i.i.d. random processes $\big( (S_i(t))_{t\ge 0}\big)_{i\ge 1} $ with law $(S_{t}^{\frac 12,+})_{t\ge 0} $. It is then easily seen that $\bar S_n(t):=S_n(S_{n-1}(...S_1(t))) $ is a standard stable subordinator of index $\beta=\frac1{2^n} $. An immediate induction indeed gives:
\begin{eqnarray*}
\forall \lambda \ge 0,\ \E[\exp(-\lambda S_n(S_{n-1}(... S_1(t))))]&=&\E[\exp(-\lambda^{1/2} S_{n-1}(...S_1(t)))]=\cdots\\
&=&\exp(-\lambda^{1/(2^n)}t).
\end{eqnarray*}
Assume that identity \eqref{ID_LAW} holds for a given $n\in \N^*$ and let $B^{n+1}$ be an additional independent Brownian motion independent of the $(B^i)_{i\in \leftB 1 ,n\rightB} $. Write:
\begin{equation}
\label{INV_INDUC}
\begin{split}
\P[\bar S_{n+1}(t)> T]&=\P[S_{n+1}(\bar S_n(t))> T]=\P[\sup_{s\in [0,T]}B^{n+1}(s)\le \frac{\bar S_n(t)}{\sqrt 2}]\\
&=\P[\bar C_n(\sqrt 2\sup_{s\in [0,T]}B^{n+1}(s))\le t],
\end{split}
\end{equation}
where $\bar C_n $ stands for the inverse of $\bar S_n$. Now, the induction hypothesis exactly gives that for all $u\ge 0$
$$\bar C_n(u)\overset{({\rm law})}{=} \sqrt 2| B^1(\sqrt 2|B^2(\sqrt 2| B^3(\cdots(\sqrt 2 |B^n(u)|)\cdots)|)|)|.$$
The proof then follows from the L\'evy identity and \eqref{INV_INDUC} since $\bar S_{n+1} $ is a subordinator of index $1/2^{n+1} $.


%



\end{proof}
\begin{REM}
It can be shown similarly that for given $(\beta_i)_{i\in \leftB 1,n\rightB}\in (0,1)^n,\ n\in \N^*$, if $\big( (S_i 
(t))_{t\ge 0}\big)_{i\in \leftB 1, n\rightB}$ are independent subordinators s.t for $i\in \leftB 1,n\rightB,\  (S_i
(t))_{t\ge 0}\overset{{\rm (law)}}{=} (S_t^{\beta_i,+})$, then 
 $\bar S_n(t):=S_n(S_{n-1}(...S_1(t)) $ is a subordinator of index $\beta:=\prod_{i=1}^n \beta_i $ and
 $$Z_T^\beta \overset{{\rm (law)}}{=}Z_{Z_{\ddots Z_{T}^{\beta_1}}^{\beta_{n-1}}}^{\beta_n} .$$ The delicate part is that this relation cannot be easily exploited from the practical viewpoint, except if $\beta_i=\frac 12, \forall i\in \leftB 1,n\rightB $.
\end{REM}

In particular, it follows from Proposition \ref{ID_LOI_PUISSN} that, for a given spatial Markov process $(X_t)_{t\ge 0}$ with generator $L $ of the form \eqref{SPATIAL_OP},  the solution of equation \eqref{CAUCHY_I} with $\beta=2^{-n} $ writes for a continuous function $f$ and for all $(T,x)\in \R^+\times \R^d$:
$$u(T,x):=\E[f(X_{Z_T^{2^{-n}}}^{0,x})]:=\E[f(X_{\sqrt 2|B^1(\sqrt 2 |B^2(\sqrt 2|B^3(\cdots(\sqrt 2|B^n(T)|)\cdots)|)|)|}^{0,x})],$$
generalizing Theorem 2.2 in Beghin and Orsingher \cite{begh:orsi:09}.

We now state the Lemma concerning the required asymptotics for the laws of $Z_T^\beta,Z_T^{\beta,h} $.
\begin{LEMME}[Bounds for the density of the inverse subordinator and its discrete approximation]
\label{BOUNDS_P_Z_P_Z_H}
For all $T>0 $,
\begin{eqnarray}
\label{CTR_BOUNDS_P_Z_P_Z_H}
p_{Z^\beta}(T,u)\le \theta(T,u),\ \forall i\in \N^*, \P[{Z_T^{\beta,h}}=t_i]\le C_\beta h \theta(T,t_i),
\end{eqnarray}
where: 
\begin{equation}
\label{DEF_THETA}
\theta(T,u):=\frac{c_\beta}{T^\beta}\exp\left( -c_{\beta}^{-1}\left\{\frac{u}{T^\beta} \right\}^{1/(1-\beta)}\right),
\end{equation} 
for some constants $c_\beta,\ C_\beta\ge 1$.
\end{LEMME}
\begin{REM} The above control is the generalization of the well known Gaussian bound for $\beta=1/2 $ recalled in Proposition \ref{ID_LOI_PUISSN}.
\end{REM}
\begin{proof}
Let us first recall the following asymptotics for $p_{S^{\beta,+}}\big(1,\cdot\big)$ from equation (2.2), (2.3) in \cite{hahn:koba:umar:11}, see also \cite{zolo:92}:
\begin{equation}
\begin{split}
p_{S^{\beta,+}}\big(1,v\big) &\sim K(\beta/v)^{(1-\beta/2)/(1-\beta)}\exp\left( -[1-\beta](v/\beta)^{\beta/(\beta-1)} \right), \ v\rightarrow 0^+,\\
& \ K=1/\sqrt{2\pi\beta(1-\beta)}, \\
p_{S^{\beta,+}}\big(1,v\big) &\sim \frac{\beta}{\Gamma(1-\beta)}v^{-\beta-1},\ v\rightarrow +\infty.
\end{split}
\label{ASYMP_ZOLO}
\end{equation}
Write now from \eqref{DENS_INV_SUB} recalling the identity $S_u^{\beta,+}\overset{({\rm law})}{=} u^{1/\beta}S_1^{\beta,+} , u>0$:
\begin{eqnarray}
p_{Z^\beta}(T,u)&=&-\partial_u[\int_0^T\frac{1}{u^{1/\beta}}p_{S^{\beta,+}}\left(1,\frac{y}{u^{1/\beta}} \right)]dy\notag\\
&=&\frac{1}{\beta u^{1+1/\beta}}\left\{\int_0^T p_{S^{\beta,+}}\left(1,\frac{y}{u^{1/\beta}}\right) dy\right.
\left. +\int_0^T \partial_yp_{S^{\beta,+}}\left(1,\frac{y}{u^{1/\beta}}\right)   ydy\right\} \notag\\
&=&
\frac T{\beta u^{1+1/\beta}}p_{S^{\beta,+}}\bigg(1,\frac{T}{u^{1/\beta}}\bigg),\label{REL_DENS}
\end{eqnarray}
integrating by parts and using \eqref{ASYMP_ZOLO}. 

Hence,
\begin{eqnarray*}
p_{Z^\beta}(T,u)&\sim& \frac{T^{-\beta}}{\Gamma(1-\beta)},\ u\rightarrow 0^+,\\
p_{Z^\beta}(T,u)&\sim& \left( \frac \beta T\right)^{\beta/(2(1-\beta))}K u^{\frac{\beta-1/2}{1-\beta}}\exp\left(-[1-\beta]\left(\beta  \right)^{\beta/(1-\beta)} 
\left(\frac{u}{T^\beta}\right)^{1/(1-\beta)}\right).\\
\ u&\rightarrow &+\infty
\end{eqnarray*}
The above controls and elementary computations give the first bound in \eqref{CTR_BOUNDS_P_Z_P_Z_H}. Indeed, there exists $c_0:=c_0(\beta)\ge 1$ s.t., 
\begin{eqnarray*}
p_{Z_T^\beta}(T,u) &\le& c\frac{T^{-\beta}}{\Gamma(1-\beta)},\  u\le  \left(\frac{T}{c_0} \right)^\beta,\\
p_{Z_T^\beta}(T,u) &\le & cK     \left( \frac \beta T\right)^{\beta/(2(1-\beta))}
       u^{\frac{\beta-1/2}{1-\beta}}\exp\left(-[1-\beta]\left(\beta  \right)^{\beta/(1-\beta)} \left\{ \frac{u}{T^\beta}  \right\}^{1/(1-\beta)}\right)\\
       &\le & \frac{c_\beta}{T^\beta}\exp\left(-c_\beta^{-1}\left\{\frac{u}{T^\beta} \right\}^{1/(1-\beta)}\right),\  u\ge (c_0T)^\beta.
\end{eqnarray*}
For the last bound we observe that for $\beta\le  \frac12$ and $u \ge (c_0T)^\beta$, $
[{u^{\frac{\beta-\frac 12}{1-\beta}}}]
 /[{T^{\frac{\beta}{2(1-\beta)}}}]\le \bar cT^{\frac{\beta}{1-\beta}([\beta-\frac 12]-\frac12)}=\bar c T^{-\beta}$, with a constant  $\bar c:=\bar c(\beta)$. For $\beta > \frac 12$, we obtain similarly that $
 [{u^{\frac{\beta-\frac 12}{1-\beta}}}]/[{T^{\frac{\beta}{2(1-\beta)}}}]$ $\le T^{-\beta}  [\frac{u}{T^\beta}]^{\frac{\beta-\frac 12}{1-\beta}}$. The last contribution can be absorbed by the exponential so that we are back to the previous case.
For $u\in [(c_0^{-1}T)^\beta,(c_0T)^\beta] $, the statement follows from \eqref{REL_DENS} and the positivity of $p_{S^{\beta^+}}(1,.)$ on compact sets.

From the control  $p_{Z^\beta}(T,u)\le \theta(T,u) $ and the relation \eqref{OBS_DIS_INV}, 
 we indeed derive that for all $i\in \N^* $, $\P[Z_T^{\beta,h}=t_i]\le C_\beta h \theta(T,t_i) $ for some $C_\beta\ge 1$. 
\end{proof}

For the rest of the analysis we will need some controls on the p.d.f. of $X_t$ and their time derivatives in both the diffusive and strictly stable case. Such results can be found in Friedman \cite{frie:64} under \A{A${}_D$} or \cite{kolo:00} under \A{A${}_S$}. For the reader's convenience we also recall in Appendix \ref{ASYMP_SMALL_TIMES} some points about the parametrix expansion from which these controls are derived. Let us emphasize that, for our analysis, we need to specificy explicitly the dependence of the constants in time since we are led to integrate over arbitrarily large time intervals in \eqref{DENS_GLOB_MARKOV}.

\begin{LEMME}[Non degenerate heat kernel bounds]
\label{HKUSUEL}
Under Assumption \A{A${}_D $} (diffusive case), there exists $(c,C):=(c,C)($\A{A${}_D$}$,d)\ge 1 $ s.t. for all $(u,x,y)\in \R_{+}^*\times (\R^d)^2 $:
\begin{eqnarray}
\label{HK_DIFF}
p(u,x,y)&\le& C\exp(C u^{1/2})g_{c}(u,x-y) 
,\nonumber\\
|\partial_u p(u,x,y)|&\le& C\frac{\exp(C u^{1/2})}{u}g_{c}(u,x-y)
,\nonumber \\
g_{c}(u,x-y)&:=&\frac{1}{(2\pi c u)^{d/2}}\exp(-\frac{|x-y|^2}{2 c u}).
\end{eqnarray}
Under Assumption \A{A${}_S$} (strictly stable case), there exists $C:=C($\A{A${}_S$}$,d)\ge 1,\ \omega:=\frac1\alpha \wedge 1 $ s.t. for all $(u,x,y)\in \R_+^{*}\times (\R^d)^2$:
\begin{eqnarray}
\label{HK_BOUND_STABLE}
p(u,x,y)&\le& C \exp(Cu^{\omega}) p_S(u,x-y) 
,\nonumber \\
\partial_u p(u,x,y)&\le&  C\frac{\exp( C u^\omega)}{u}p_S(u,x-y),\nonumber \\
p_S(u,x-y)&=&\frac{c_\alpha}{u^{d/\alpha}}\left(1+\frac{|x-y|}{u^{1/\alpha}}\right)^{-(d+\alpha)},
\end{eqnarray}
where the constant $c_\alpha$ is chosen to have $\int_{\R^d}p_S(u,z)dz=1 $.

Observe that the above controls are exponentially explosive in time. 
\end{LEMME}

We now give a fundamental lemma for the error analysis. The quantities below derive from the time sensitivities or the previous error bounds from \cite{kona:mamm:00}, \cite{kona:mamm:02} 
used to control the terms ${\mathcal E}_1, {\mathcal E}_2 $ in \eqref{ERROR_1} respectively.
\begin{LEMME}[Fundamental Bounds]
\label{GROS_LEMME}
Let $T>0$ be given.  The following controls hold:
\begin{trivlist}
\item[-] \textbf{Diffusive case}: There exists $\bar c:=\bar c($\A{A${}_D$}) s.t. for $\gamma \in \{0, \frac 12 ,1 \} $ and for all $(x,y)\in (\R^d)^2,\ x\neq y $:
\begin{eqnarray}
\label{EQ_GROS_LEMME_DIFF}
\int_{0}^{+\infty} 
u^{-\gamma} 
p_{Z^\beta}(T,u) g_{c}(u,x-y)\exp(C u^{1/2})du\nonumber\\
+ \sum_{i\ge 1} t_i^{-\gamma} 
\P[Z_T^{\beta,h}=t_i] g_{c}(t_i,x-y)\exp(C t_i^{1/2})\nonumber\\
\le \bar c \left\{ \Big( \frac{\I_{d\ge 3}}{|x-y|^{2\gamma}}+\frac{\I_{d=2}}{|x-y|^{\I_{\gamma=1}}T^{\frac\beta 2 \I_{\gamma>0}}}+\frac{\I_{d=1}}{|x-y|^{\I_{\gamma>0}}T^{\frac \beta 2 \I_{\gamma=1}}} \Big)\hat p_\beta(T,x-y)\right. \notag\\
\left.+\frac{1}{T^{\beta \gamma}}\tilde p_\beta(T,x-y)  \right\}.
\end{eqnarray}

For the Markov Chain approximation, for all $\varepsilon \in (0,\frac 1 5) $,  there exists $c_\varepsilon:=c_\varepsilon($\A{A${}_{D,m}$}) s.t. denoting for all $(x,y)\in (\R^d)^2$,
$$q_m(t_i,x-y):=\frac{c_m}{t_i^{d/2}}\left(1+\frac{|x-y|}{t_i^{1/2}}\right)^{-m},$$ 
where $c_m$ is s.t. $\int_{\R^d}q_m(t_i,z) dz=1 $,
we have for all $x\neq y$:
\begin{eqnarray}
\label{EQ_GROS_LEMME_CDM}
\sum_{i\ge \lceil h^{-4/5-\varepsilon} \rceil } q_m(t_i,x-y)t_i^{-1/2}\exp(Ct_i^{1/2})\P[Z_T^{\beta,h}=t_i]\nonumber\\
\le c_\varepsilon \bigg\{\bigg\{\frac{\I_{d=1}+\I_{d\ge 3}}{|x-y|}+\frac{\I_{d=2}}{T^{\beta/2}}\bigg\}
\hat q_{m-[(d-1)+\I_{d=1}],\beta}(T,x-y) 
+\frac{1}{T^{\beta/2}}\tilde q_{m,\beta}(T,x-y) \bigg\}\nonumber\\
:=c_\varepsilon {\mathcal E}_{\beta,{\rm Space, LLT}}^M(T,x-y).
\end{eqnarray} 
\item[-] \textbf{Strictly Stable case}: There exists $\bar c:=\bar c($\A{A${}_S$}) s.t. for $\gamma \in \{0,1 \} $ and for all $(x,y)\in (\R^d)^2,\ x\neq y $:
\begin{eqnarray}
\label{EQ_GROS_LEMME_STABLE}
\int_{0}^{+\infty} 
u^{-\gamma} 
p_{Z^\beta}(T,u) p_S(u,x-y)\exp(C u^{\omega})du
\\
+\sum_{i\ge 1} t_i^{-\gamma} 
\P[Z_T^{\beta,h}=t_i] p_{S}(t_i,x-y)\exp(C t_i^{\omega})\nonumber 
\\
\nonumber \le \bar c
\Big\{ \Big( \frac{\I_{|x-y|\le T^{\beta/\alpha}}}{|x-y|^{\gamma \alpha}}+\frac{\I_{|x-y|>T^{\beta/\alpha}}}{T^{\beta \gamma }}\Big) \hat p_\beta(T,x-y)+\frac{1}{T^{\beta\gamma }}\tilde p_\beta(T,x-y) \Big\}.
\end{eqnarray}
\end{trivlist}
\end{LEMME}

As a direct corollary of equation \eqref{DENS_GLOB_MARKOV} and Lemmas 
\ref{HKUSUEL} and \ref{GROS_LEMME} taking $\gamma=0 $ we obtain some controls
 on the density of $X_{Z_t^{\beta}}$. Precisely, we have:
\begin{COROL}[Some density estimates for the spatial motion along the inverse subordinator]
\label{LEMME_QUI_FAIT_A_LA_KOCHUBEI}
There exists $c:=c(\A{A},\beta,\alpha)\ge 1$ s.t. for any fixed $T>0$ the following estimates hold:
\begin{eqnarray*}
p_{X_{Z_T^\beta}}(x,y)\le c (\hat p_\beta+\tilde p_\beta)(T,x-y),
\end{eqnarray*}
in both the diffusive and strictly stable case with the associated definitions of $\hat p_\beta,\tilde p_\beta $ in Theorems \ref{THM_DIFF} and \ref{THM_STABLE}.
\end{COROL}
\begin{REM}\label{RQ_DIAG_FRAC}
In the above bounds, from the definitions of $\hat p_\beta,\tilde p_\beta $, we have in some sense the \textit{usual scaling} in time for the Gaussian and stable heat-kernels at time $T^\beta $. We indeed observe that the dichotomy between the diagonal and off-diagonal regime is performed depending on $|x-y|\le K(T^\beta)^{1/\alpha} $ or $|x-y|> K(T^\beta)^{1/\alpha} $ for some positive constant $K$. In the diagonal regime we also have, as indicated previously, an additional spatial singularity
whereas 
in the off-diagonal one, the concentration is modified in the diffusive case due to the integration of the density $p_{Z^\beta}(T,u)$ for large $u$ in \eqref{DENS_GLOB_MARKOV}. 

In the Gaussian case, those bounds are coherent with those of \cite{koch:eide:04} and our approach can be viewed as an alternative probabilistic
strategy to derive estimates for the fundamental solution of fractional in time diffusive PDEs. 
For the strictly stable case, to the best of our knowledge these results seem to be new. 
Some lower bounds are derived in Appendix \ref{APP_HK}.
\end{REM}

\begin{proof}[Proof of Lemma \ref{GROS_LEMME}]
\hspace*{1cm}
\begin{trivlist}
\item[-] \textbf{Diffusive case}. 
We focus on the first integral in the l.h.s. of \eqref{EQ_GROS_LEMME_DIFF} since the other term can be handled similarly (Riemann sum approximation).  Exploiting the bounds from Lemma \ref{BOUNDS_P_Z_P_Z_H},  we
split the time integration domain at the characteristic time scale $T^\beta $ and write:
\begin{eqnarray}
\label{DEC_1}
Q_{\gamma,\beta}(T,x-y):=\int_{0}^{+\infty}
u^{-\gamma} 
p_{Z^\beta}(T,u) g_{c}(u,x-y)\exp(C u^{1/2})du
\\
        \le \frac{c_\beta 
        \exp(CT^{\beta/2})}{T^\beta}\exp\left(-\frac{c^{-1}}4\frac{|x-y|^2}{T^\beta}\right)\int_{0}^{T^\beta}\frac{du}{u^{\gamma+d/2}}\exp\left(-\frac{c^{-1}}4\frac{|x-y|^2}{u}\right)\notag\\
        +\frac{c_\beta}{T^{\beta}}\int_{T^\beta}^{+\infty}\frac{du}{u^{\gamma+d/2}} 
        \exp(Cu^{1/2})\exp\left(-c^{-1}\frac{|x-y|^2}{2u}\right)\exp\left(-c_\beta^{-1}\left\{ \frac{u}{T^{\beta}}\right\}^{1/(1-\beta)}  \right)\notag\\
        :=Q_{\gamma,\beta,s}(T,x-y)+Q_{\gamma,\beta,l}(T,x-y),\notag
\end{eqnarray}
where the subscripts $s,l$ respectively denote the contributions in \textit{short} and \textit{long} time.
For the contribution $Q_{\gamma,\beta,s}(T,x-y)$, we have to analyze the integral $I(T^\beta,\gamma,|x-y|,d):=\int_{0}^{T^\beta}\frac{du}{u^{\gamma+d/2}}\exp(-\frac{c^{-1}}4\frac{|x-y|^2}{u})$. Observe that except if $d=1,\gamma=0$ there is a singularity in small time in the integral which needs to be compensated by the spatial component $|x-y|\neq 0 $ in the exponential. This yields the spatial (diagonal) singularity.

To equilibrate the time singularity we will thoroughly rely on the fact that
the density of  the hitting time of level $a>0$ for the scalar Brownian motion writes (see e.g. \cite{revu:yor:99}):
\begin{equation}
\label{HIT_TIME_BM}
\forall u>0,\ f_a(u):= \frac{2a}{(2\pi u^{3})^{1/2}}\exp\left(-\frac{a^2}{2u} \right).
\end{equation}

\begin{trivlist}
\item[-]If $\gamma+ d/2\ge 3/2$ we write for all $|x-y|\neq 0 $:
\begin{eqnarray*}
I(T^\beta,\gamma,|x-y|,d)\\
= \int_{0}^{T^\beta}\frac{du |x-y| }{u^{3/2}} \frac{|x-y|^{d+2\gamma-3}}{u^{(d+2\gamma-3)/2}}\exp\left(-\frac{c^{-1}}4\frac{|x-y|^2}{u}\right)\frac{1}{|x-y|^{d+2\gamma-2}}\\
\le\frac{\tilde c}{|x-y|^{d+2(\gamma-1)}}\int_0^{T^\beta }\frac{du |x-y| }{(2\pi u^3)^{1/2} } \exp\left(-\tilde c^{-1}\frac{|x-y|^2}{2u}\right)\le \frac{\tilde c}{|x-y|^{d+2(\gamma-1)}}.
\end{eqnarray*}
\item[-] If $\gamma+d/2<3/2 $ which actually happens for $d=1, \gamma\in \{0,\frac 12\} $ or $d=2$ and $\gamma=0 $ we write:
\begin{eqnarray*}
I(T^\beta,0,|x-y|,1)&\le & c T^{\beta/2},\\
I(T^\beta,0,|x-y|,2)&=&I(T^\beta,1/2,|x-y|,1)\le \frac{T^{\beta/2}}{|x-y|} \int_{0}^{T^\beta}\frac{du |x-y| }{u^{3/2}} \exp\left(-\frac{c^{-1}}{4}\frac{|x-y|^2}{u}\right)\\
&\le&\frac{\tilde c T^{\beta/2}}{|x-y|}.
\end{eqnarray*}
\end{trivlist}

The previous results can be summarized as follows. There exists $c:=c($\A{A${}_D $},$\beta)$ s.t. for all $(T,x,y)\in \R_+^*\times (\R^d)^2$, $|x-y|\neq 0 $: 
\begin{equation}
\label{CTR_DIFF_S}
\begin{split}
Q_{\gamma,\beta, s}(T,x-y)\le c\exp(cT^{\beta/2})\exp\left(-c^{-1}\frac{|x-y|^2}{2T^\beta}\right)\\
\times \Big\{\frac{\I_{d=1,\gamma=0}}{T^{\beta/2}}+\frac{\I_{d=2,\gamma=0}+\I_{d=1,\gamma=1/2}}{T^{\beta/2}|x-y|} +\frac{\I_{d+2\gamma \ge 3}}{T^\beta |x-y|^{d+2(\gamma-1)}}\Big\}.
\end{split}
\end{equation}
The restriction $|x-y|\neq 0 $ is not necessary if $d=1,\ \gamma=0$.

On the other hand, from Young's inequality, there exists $\bar c_\beta>0 $ s.t. for all $u,\varepsilon>0 $:
$$ u^{1/2}=\left( \varepsilon\frac{u}{T^\beta}\right)^{1/2}(\varepsilon^{-1}T^\beta)^{1/2}\le \bar c_\beta \left\{ \left(\varepsilon \frac{u}{T^\beta}\right)^{1/(1-\beta)} 
+\varepsilon^{-1/(1+\beta)} T^{\beta/(1+\beta)}\right\}.$$
We thus get from \eqref{DEC_1} that, for $\varepsilon $ sufficiently small, there exists $\tilde c_\beta$ s.t.:
\begin{eqnarray*}
Q_{\gamma,\beta, l}(T,x-y)\le \frac{cc_\beta}{T^{\beta(1+\gamma+d/2)}}\exp(\tilde c_\beta T^{\beta/(1+\beta)})\\\
\times \int_{T^\beta}^{+\infty}du\exp\left(-c^{-1}\frac{|x-y|^2}{2u}\right)\exp\left(-\frac{c_\beta^{-1}}2\left\{ \frac{u}{T^{\beta}}\right\}^{1/(1-\beta)}\right).
\end{eqnarray*}
Now two cases occur according to the dichotomy indicated in Remark \ref{RQ_DIAG_FRAC}. If $|x-y|\le T^{\beta/2} $ then setting $\tilde u:=u/T^\beta $ we get $Q_{\gamma,\beta, l}(T,x-y)\le \frac{\bar c_\beta}{T^{\beta(\gamma+ d/2)}}\exp(\bar c_\beta T^{\beta/(1+\beta)})$. If now $|x-y|> T^{\beta/2} $, we see that the two exponential contributions in the integral are \textit{equivalent} if  $ \kappa_1 \{ |x-y|^2T^{\beta/(1-\beta)}\}^{\frac{1-\beta}{2-\beta}}\le  u  \le \kappa_2 \{ |x-y|^2T^{\beta/(1-\beta)}\}^{\frac{1-\beta}{2-\beta}},\ 1<\kappa_1<\kappa_2<+\infty$, for which 
\begin{eqnarray*}
\exp\bigg(-\frac{c_\beta^{-1}}2\left\{ \frac{u}{T^{\beta}}\right\}^{1/(1-\beta)}\bigg)\\
\le \exp\bigg(-\frac{c_\beta^{-1}\kappa_1^{1/(1-\beta)} }{4}\bigg[\frac{|x-y|^2}{T^\beta}\bigg]^{1/(2-\beta)}\bigg)\exp\bigg(-\frac{c_\beta^{-1}}4\left\{ \frac{u}{T^{\beta}}\right\}^{1/(1-\beta)}\bigg) .
\end{eqnarray*}
On the other hand, for $ u\in [T^\beta, \kappa_1\{ |x-y|^2T^{\beta/(1-\beta)}\}^{\frac{1-\beta}{2-\beta}}]$ we have:
$$\exp\bigg(-c^{-1}\frac{|x-y|^2}{2u}\bigg)\le \exp\bigg(-(2c\kappa_1)^{-1}\bigg[\frac{|x-y|^2}{T^\beta}\bigg]^{1/(2-\beta)}\bigg) .$$
For $u\in [\kappa_2\{ |x-y|^2T^{\beta/(1-\beta)}\}^{\frac{1-\beta}{2-\beta}},+\infty)$, it  suffices to bound 
$$\exp\bigg(-\frac{c_\beta^{-1}}4\left\{ \frac{u}{T^{\beta}}\right\}^{1/(1-\beta)}\bigg)\le  \exp\bigg(-\frac{c_\beta^{-1}}4\kappa_2^{1/(1-\beta)}\left( \frac{|x-y|^2}{T^\beta}\right)^{1/(2-\beta)}\bigg).$$
Thus, setting as in the diagonal case $\tilde u=u/T^\beta $, we derive that there exists $\bar c_\beta:=\bar c_\beta($\A{A${}_D$},$\kappa_1,\kappa_2)\ge 1$, s.t. for all $(T,x,y)\in \R_{+}^*\times (\R^d)^2 $:
\begin{equation}
\label{CTR_DIFF_L}
Q_{\gamma,\beta,l}(T,x-y)\le \frac{\bar c_\beta}{T^{\beta( \gamma+d/2)}}\exp(\bar c_\beta T^{\beta/(1+\beta)})\exp\left(-\bar c_\beta^{-1}\bigg[\frac{|x-y|^2}{T^\beta}\bigg]^{1/(2-\beta)}\right).
\end{equation}
The statement \eqref{EQ_GROS_LEMME_DIFF} follows from \eqref{CTR_DIFF_L} and \eqref{CTR_DIFF_S}.
\item[-] \textbf{Markov Chain.}

Setting $i_C:=\lceil h^{-4/5-\varepsilon}\rceil,\ \ui:=\lceil  T^\beta/h\rceil $, let us write from Lemma \ref{BOUNDS_P_Z_P_Z_H}:
\begin{eqnarray}
\label{CTR_E122}
Q_{\frac12,\beta}(T,x-y):=
\sum_{i\ge i_C}^{}\P[Z_T^{\beta,h}=t_i] \exp(Ct_i^{1/2})t_i^{-1/2}q_m(t_i,x-y)   \nonumber\\
\le \frac{c}{T^\beta}\exp(CT^{\beta/2})\int_{t_{i_C}}^{t_{\ui}} u^{-(d+1)/2}\left(1+ \frac{|x-y|}{u^{1/2}}\right)^{-m}du \nonumber \\
+ \frac{c}{T^{\beta(1+(d+1)/2)}}\exp(cT^{\beta/(1+\beta)})\int_{t_{\ui}}^{+\infty} \left(1+ \frac{|x-y|}{u^{1/2}}\right)^{-m} \exp(-c_\beta^{-1}\left\{\frac{u}{T^\beta} \right\}^{1/(1-\beta)})du\nonumber\\
:=Q_{\frac12,\beta,s}(T,x-y)+Q_{\frac12,\beta,l}(T,x-y),
\end{eqnarray}
using as above Young's inequality to get the upper bound and replacing, up to modifications of the constants, the Riemann sums by integrals.

We now derive setting $\tilde t=u/|x-y|^2 $ and up to a modification of $c$:
\begin{eqnarray}
Q_{\frac12,\beta,s}(T,x-y) \notag\\
\le \frac{c}{T^\beta}\exp(cT^{\beta/2})\frac{t_{\ui}^{\frac 12\I_{d=1}}}{|x-y|^{d-1+\I_{d=1}}}\int_{t_{i_C}/|x-y|^2}^{t_{\ui}/|x-y|^2} \tilde t^{-(d+1+\I_{d=1})/2}\left(1+ \frac{1}{\tilde t^{1/2}}\right)^{-m
}d\tilde t \nonumber \\
\le \frac{c}{T^{\beta}}\exp(cT^{\beta/2})\frac{t_{\ui}^{\frac 12\I_{d=1}}}{|x-y|^{d-1+\I_{d=1}}}\int_{t_{i_C}/|x-y|^2}^{t_{\ui}/|x-y|^2} \tilde t^{-3/2}\left(1+ \frac{1}{\tilde t^{1/2}}\right)^{-m+ (d-2+\I_{d=1}) 
}d\tilde t \nonumber\\
\le  \frac{c}{T^\beta}\exp(cT^{\beta/2}) \left[  \frac{T^{\beta/2}\I_{d=1}}{|x-y|} +\frac{\I_{d\ge 2}}{|x-y|^{d-1}}\right]\times \left(1+\frac{|x-y|}{T^{\beta/2}} \right)^{-m+(d-1)+\I_{d=1}}.\nonumber\\
\label{CTR_Q12_S}
\end{eqnarray}

On the other hand if $|x-y|\le T^{\beta/2} $ (diagonal regime) we readily get $Q_{\frac12,\beta,l}(T,x-y)\le \frac{c\exp(cT^{\beta/(1+\beta)})}{T^{\beta(d+1)/2}} $. If now $|x-y|>T^{\beta/2} $, splitting as above for $u\in [t_{\ui}, \kappa_1\{ |x-y|^2T^{\beta/(1-\beta)}\}^{\frac{1-\beta}{2-\beta}}]$ and $u\ge \kappa_1\{ |x-y|^2T^{\beta/(1-\beta)}\}^{\frac{1-\beta}{2-\beta}} $ yields :
\begin{eqnarray}
\label{CTR_Q12_L}
Q_{\frac12,\beta,l}(T,x-y)\le \frac{c}{T^{\beta(d+1)/2}} \exp(cT^{\beta/(1+\beta)})\left(1+\left[\frac{|x-y|}{T^{\beta/2}}\right]^{1/(2-\beta)} \right)^{-m}.\nonumber\\
\end{eqnarray}
Equation \eqref{EQ_GROS_LEMME_CDM} follows putting \eqref{CTR_Q12_L} and \eqref{CTR_Q12_S} in \eqref{CTR_E122}.
 \item[-] \textbf{Strictly Stable Case.}
We again focus on the first term in the l.h.s. of \eqref{EQ_GROS_LEMME_STABLE}. The discrete sum can be handled similarly.
Analogously to  \eqref{DEC_1}, replacing $g_{c} $ by $p_S $ and the exponent $1/2$ in the exponential by $\omega $, we write from Lemmas \ref{BOUNDS_P_Z_P_Z_H} and \ref{HKUSUEL}:
\begin{eqnarray*}
Q_{\gamma,\beta,s}(T,x-y)&\le &\frac{c_\beta c_\alpha
\exp(CT^{\beta\omega})}{T^\beta}\int_{0}^{T^\beta}\frac{du}{u^{\gamma+d/\alpha}}\frac{1}{(1+\frac{|x-y|}{u^{1/\alpha}})^{d+\alpha}}.
\end{eqnarray*}
Two cases are to be distinguished.
\begin{trivlist}
\item[-] If $d=1,\gamma=0, \alpha>1$ then the above diagonal singularity is integrable and one gets:
\begin{eqnarray}
Q_{0,\beta,s}(T,x-y)&\le &\frac{c_\beta c_\alpha \exp(C T^{\beta\omega})}{T^\beta}\frac{1}{(1+\frac{|x-y|}{T^{\beta/\alpha}})^{1+\alpha}}\int_{0}^{T^\beta}\frac{du}{u^{1/\alpha}}\notag\\
&\le &\frac{c c_\alpha\exp(cT^{\beta\omega})}{T^{\beta/\alpha}}\frac{1}{(1+\frac{|x-y|}{T^{\beta/\alpha}})^{1+\alpha}}.\label{DIAG_S_PART}
\end{eqnarray}
In particular, this estimate holds even for $x=y$. There is no spatial singularity.
\item[-] For all the other cases, i.e.  $d=1,\gamma =1, \alpha>1 $ or $d=1,\alpha\le 1$, $d\ge 2 $ for $\gamma\in \{0,1\} $, 
we need to consider $x\neq y $ in order to equilibrate the time singularity. In small time we still distinguish the diagonal and off-diagonal regimes. Precisely,
\begin{trivlist}
\item[-] If $|x-y|\le T^{\beta/\alpha} $, then:
\begin{eqnarray}
Q_{\gamma,\beta,s}(T,x-y)&\le &\frac{c_\beta c_\alpha \exp(CT^{\beta\omega})}{T^\beta}\Big\{\frac{1}{|x-y|^{d+\alpha}}\int_0^{|x-y|^\alpha} du u^{1-\gamma}+\int_{|x-y|^\alpha}^{T^{\beta}} \frac{du}{u^{\gamma+d/\alpha}}\Big\}\nonumber\\
&\le &c_\beta c\exp(cT^{\beta\omega})\Big\{\frac{1}{T^\beta |x-y|^{d-\alpha(1-\gamma)}}+\frac{1}{T^{\beta (\gamma+ d/\alpha)}} \Big\}
.\label{DIAG_S_GEN_DIAG}
\end{eqnarray}

\item[-] If $|x-y|> T^{\beta/\alpha} $
\begin{eqnarray}
Q_{\gamma,\beta,s}(T,x-y)&\le &\frac{c_\beta c_\alpha \exp(CT^{\beta\omega})}{T^\beta}\frac{1}{|x-y|^{d+\alpha}}\int_0^{T^\beta} du u^{1-\gamma}\nonumber\\
&\le &c_\beta c\exp(cT^{\beta\omega})\frac{T^{\beta(1-\gamma)}}{|x-y|^{d+\alpha}}.\label{DIAG_S_GEN}
\end{eqnarray}
\end{trivlist}
\end{trivlist}
Let us mention that, in the diagonal regime, we get as in the diffusive case an additional spatial singularity.

Let us now deal with $Q_{\gamma,\beta,l}(T,x-y) $. 
We first write from Young's inequality that there exists $\bar c_\beta$ s.t. for all $u,\varepsilon>0 $:
$$u^\omega =\left( \varepsilon \frac{u}{T^\beta} \right)^\omega (\varepsilon^{-1}T^\beta)^\omega\le \bar c_\beta\left\{ \left( \varepsilon \frac{u}{T^\beta}\right)^{1/(1-\beta)}+(\varepsilon^{-1}T^\beta)^{\omega/(1-\omega(1-\beta))} \right\}.$$
Thus, taking $\varepsilon $ small enough yields: 
\begin{eqnarray}
Q_{\gamma,\beta,l}(T,x-y)&\le & c\frac{\exp(cT^{\beta\omega/(1-\omega(1-\beta))})}{T^{\beta} }
\int_{T^\beta}^{+\infty} du (\frac{u}{|x-y|^{d+\alpha}}\wedge u^{-d\beta/\alpha})u^{-\gamma}\nonumber\\
&&\times  \exp\left(-\frac{c_\beta^{-1}}{2}\left\{\frac{u}{T^\beta}\right\}^{1/(1-\beta)}\right)
\nonumber\\
&\le&c\exp(cT^{\beta\omega/(1-\omega(1-\beta))})\left(\frac{T^\beta }{ |x-y|^{d+\alpha}}\wedge T^{-\beta d/\alpha}\right)T^{-\beta \gamma}. 
\label{HD_STABLE}
\end{eqnarray}
The result follows from \eqref{DIAG_S_PART}, \eqref{DIAG_S_GEN_DIAG}, \eqref{DIAG_S_GEN}, 
and \eqref{HD_STABLE}.
\end{trivlist}

\end{proof}

\subsection{Time Sensitivity: Proof of Lemma \ref{LEMME_T_SENS}}
The lemma, which gives the control on the time sensitivity term ${\mathcal E}_1(T,x,y)$ in the decomposition error \eqref{ERROR_1}, is proved using the  bounds  of Lemma \ref{GROS_LEMME}.
We simply write from Lemma \ref{HKUSUEL}, and Lemma \ref{BOUNDS_P_Z_P_Z_H} for the first time step, that in the diffusive case, there exists $\bar c:=\bar c($\A{A${}_D$}) s.t. :
\begin{eqnarray*}
|{\mathcal E}_1(T,x,y)|\\
\le \bar c \left\{\frac{h\exp(Ch^{1/2})}{T^\beta |x-y|^d}\exp\left(-c^{-1}\frac{|x-y|^2}{h} \right) \right.\\
\left. +h \int_h^{+\infty} p_{Z^\beta}(T,u)u^{-1}\exp(Cu^{1/2})g_{c}(u,x-y)du\right\},
\end{eqnarray*}
whereas in the strictly stable case, there exists $\bar c:=\bar c($\A{A${}_S$}) s.t. :
\begin{eqnarray*}
|{\mathcal E}_1(T,x,y)|\le \bar c \left\{  \frac{\exp(Ch^\omega) h^2}{T^\beta|x-y|^{d+\alpha}}+h\int_h^{+\infty} p_{Z^\beta}(T,u)u^{-1}\exp(Cu^\omega)p_{S}(u,x-y)du\right\}.
\end{eqnarray*}
Recalling that in the strictly stable case we assumed $h^{1/\beta}\le T $ and $h^{1/\alpha}\le |x-y| $, the result then follows from the above controls and equations \eqref{EQ_GROS_LEMME_DIFF},  \eqref{EQ_GROS_LEMME_STABLE} in Lemma \ref{GROS_LEMME}.

\subsection{Sensitivity in Space: Proof of Lemmas \ref{LEMME_S_SENS_D} and \ref{LEMME_S_SENS_S}}

\subsubsection{The case of a diffusive spatial motion}  
The key control for the analysis of the term ${\mathcal E}_2(T,x,y) $ in \eqref{ERROR_1} is provided by the following Lemma. 
\begin{LEMME}[Control for the Spatial Error]
\label{GROSSE_BORNE_SPATIALE_DIFF}
For our current approximation schemes in \eqref{SCHEME} we have:
\begin{trivlist}
\item[-] Under \A{A${}_{D,{\rm Eul}}$}, i.e. for the Euler scheme,  there exists $(C,c):=(C,c)($\A{A$_{D}$}$ ,d)\ge 1$ s.t. for a given $h>0$ and for all $i\in \N^*$:
\begin{equation}
\label{DIFF_EUL_SCHEME} 
\begin{split}
|(p-p_{{\rm Eul}}^h)(t_i,x,y)|&\le  C  h\{ \frac 1{t_i^{1/2}}\vee \exp(Ct_i^{1/2})\}  g_{c}(t_i,x-y), \\
g_{c}(t_i,z)&=\frac{1}{(2\pi c t_i)^{d/2}}\exp\left( -\frac{|z|^2}{2 c t_i}\right),\ z\in \R^d,
\end{split}
\end{equation}
standing for the usual $d $-dimensional Gaussian density with variance  $c t_i$.
\item[-] Under \A{A${}_{D,m}$}, i.e. general Markov Chain approximation, 
 there exists $C:=C($\A{A${}_{D,m}$}$,d)\ge 1$ s.t. if $ t_i\ge h^{\delta}$, $\delta<1/5 $:
\begin{equation}
\label{LLT_DIFF_MARG}
\begin{split}
|(p-p^h)(t_i,x,y)|&\le  C  h^{1/2}\{ \frac 1{t_i^{1/2}}\vee \exp(Ct_i^{1/2})\}  q_m(t_i,x-y), \\
0\le q_m(t_i,x-y)&:= c_m t_i^{-d/2}\left( 1+\frac{|x-y|}{t_i^{1/2}}\right)^{-m}.
\end{split}
\end{equation}
\end{trivlist}
\end{LEMME}
Let us mention that $q_m$ enjoys the same parabolic scaling as $g_{c} $ but has polynomial decay.
We observe the usual convergence rates in $h$ and $h^{1/2} $ respectively for the Euler scheme and the Markov Chain approximation. For general Markov Chains, this is the rate of the Gaussian LLT (see \cite{petr:05}, \cite{kona:mamm:00},\cite{kona:mamm:09}). For the Euler scheme, the innovations are already Gaussian so that the first term disappears in the previously mentioned LLT, yielding a contribution of order $h$.

The important point of the previous Lemma is that it specifies the behaviour of the constants in \textit{short} and \textit{long} time.
This allows us to balance, as for the time sensitivity, those constants with  the controls on the laws of $Z_T^\beta,\ Z_T^{\beta,h}$ given in Lemma \ref{BOUNDS_P_Z_P_Z_H}  
thanks to Lemma \ref{GROS_LEMME}.

%
\begin{proof}
For the Euler scheme, the short time behaviour of the error has already been established in \cite{gobe:laba:08} (see Theorem 2.3 therein) whereas the exponential bounds can once again be derived from the explicit form of the parametrix series used to investigate the error in \cite{kona:mamm:02}.  
For the sake of completeness, we recall these aspects in Appendix \ref{APP_DIFF}. 

For  a general Markov Chain approaching the diffusion, i.e. for which the innovations are not necessarily Gaussian, to enter a Gaussian asymptotics specified by the Gaussian LLT, a certain number of time steps is needed. This is why we impose the condition $ t_i\ge h^{\delta}$, $\delta<1/5 $ which is the one required in \cite{kona:mamm:09} (see assumption (B2) therein) to establish the Gaussian LLT in short time.
The behaviour of the constants in large time can also be derived from \cite{kona:mamm:00} (similarly to the procedure presented in Appendix \ref{APP_DIFF}). For the control in short time we refer to Remark 1 after Theorem 1 in \cite{kona:mamm:09}. Anyhow, for the reader's convenience, we recall in Appendix \ref{APP_DIFF} 
the approach developed therein.
\end{proof}



We thus directly get from \eqref{ERROR_1} 
and  \eqref{DIFF_EUL_SCHEME} in Lemma \ref{GROSSE_BORNE_SPATIALE_DIFF} 
 that there exists $(c,C):=c($\A{A${}_D$})$\ge 1$ s.t. for all $T>0, (x,y)\in (\R^{d})^2 $, $x\neq y$:
\begin{equation}
\label{PREAL_CTR_E12}
\begin{split}
|{\mathcal E}_{2}^{{\rm Eul}}(T,x,y)|&\le C h 
\sum_{i\ge 1}  ^{} \P[Z_{T}^{\beta,h}=t_i]\exp(Ct_i^{1/2})t_i^{-1/2}g_{c}(t_i,x-y) 
.
\end{split}
\end{equation}
The result \eqref{CTR_E12_EUL} for the Euler scheme now follows from equation \eqref{EQ_GROS_LEMME_DIFF} in Lemma \ref{GROS_LEMME}.

For the Markov Chain approximation, setting $i_C:=\lceil h^{-4/5-\varepsilon}\rceil $ we write using \eqref{LLT_DIFF_MARG}:
\begin{equation}
\label{DEC_SPACE_MARKOV}
\begin{split}
|{\mathcal E}_{2}(T,x,y)|&\le  C\bigg\{ \sum_{i\in \leftB 1,i_C\rightB} \P[Z_{T}^h=t_i]|(p-p^h)(t_i,x,y)|\\
&+ \sqrt h \sum_{i>i_C  } \P[Z_{T}^{\beta,h}=t_i] \exp(C t_i^{1/2})\frac{1}{t_i^{1/2}}q_m(t_i,x-y)\bigg\}\\
&\le c_\varepsilon\{ {\mathcal E}_{21}(T,x,y)+\sqrt h {\mathcal E}_{\beta,{\rm Space, LLT}}^M(T,x-y)\},
\end{split}
\end{equation}
exploiting equation \eqref{EQ_GROS_LEMME_CDM} in Lemma \ref{GROS_LEMME}  for the last inequality in the previous r.h.s. It therefore remains to control $ {\mathcal E}_{21}(T,x,y)$.

For this contribution, we do not have comparison results between the two densities. The technique thus consists in controlling each of the terms. Observe to this end that a simple parametrix expansion for the scheme, similar to the one performed in \cite{kona:mamm:00} using Lemmas 3.6 and  3.11 therein, 
(see also Appendix \ref{ASYMP_SMALL_TIMES}),  
would give that there exists $c:=c($\A{A${}_{D,m}$}$,d)\ge 1$ s.t. for all $i\in \leftB 1,i_C\rightB , (x,y)\in (\R^d)^2$: 
\begin{equation}
\label{CTR_DENS_SCHEME}
p^h(t_i,x,y)\le \frac{c\exp(ct_i^{1/2})}{t_i^{d/2}}\left(1+\frac {|x-y|}{t_i^{1/2}} \right)^{-2(m-1)}.
\end{equation}
We thus derive from the definition in \eqref{DEC_SPACE_MARKOV} and equations \eqref{CTR_BOUNDS_P_Z_P_Z_H}, \eqref{HK_DIFF}, \eqref{CTR_DENS_SCHEME} that:
\begin{eqnarray*}
{\mathcal E}_{21}(T,x,y)\le \frac{c}{T^{\beta}}h \exp(cT^{\beta/2}) \\
\left \{ \sum_{i=1}^{i_C} t_{i}^{-d/2}\left[\exp\left(-\frac{c^{-1}}2\frac{|x-y|^2}{t_i}  \right) +\left(1+\frac {|x-y|}{t_i^{1/2}} \right)^{-2(m-1)} \right]\right\}.
\end{eqnarray*}
From the arguments in the proof of Lemma \ref{GROS_LEMME}, we can write for all $d\ge 1 $:
\begin{eqnarray}
\label{CTR_E121}
{\mathcal E}_{21}(T,x,y)&\le &\frac{c}{T^{\beta}}\exp(cT^{\beta/2})  \frac{t_{i_C}^{\frac 12\I_{d\le 2}}
}{|x-y|^{d-2+\I_{d\le2}}}\nonumber\\
&&\times\left[\exp\left(-c\frac{|x-y|^2}{t_{i_C}}  \right) +\left(1+\frac {|x-y|}{t_{i_C}^{1/2}} \right)^{-2(m-1)+(d-2)+\I_{d\le 2}} \right]\nonumber\\
&&\le {\mathcal E}_{\beta,{\rm Space, NoLLT}}^{M,\varepsilon}(T,x-y,h),
\end{eqnarray}
recalling $t_{i_C}=h^{1/5-\varepsilon}$.
Putting together estimates \eqref{CTR_E121} into \eqref{DEC_SPACE_MARKOV} for the Markov Chain approximation 
concludes 
the proof of Lemma \ref{LEMME_S_SENS_D}.

\subsubsection{The case of strictly stable driven SDE}

Similarly to the diffusive case, the crucial point is the following Lemma.
\begin{LEMME}\label{LE_LEMME_SPAT_STABLE}
Under \A{A${}_{S} $} there exists  $C:=C($\A{A${}_{S}$}), s.t. for a given $t_i:=ih$ we have:
\begin{equation}
\label{CTR_DIFF_S_TIME}
\begin{split}
|(p-p_{{\rm Eul}}^h)(t_i,x,y)|\le Ch\{ \frac{1}{t_i}\vee \exp(Ct_i^\omega)\}p_S(t_i,x-y),\\
p_S(t_i,z):=\frac{c_\alpha}{t_i^{d/\alpha}} \frac{1}{\left(1+\frac{|z|}{t_i^{1/\alpha}} \right)^{d+\alpha}},\ \forall z\in \R^d,\ \int_{\R^d} p_{S}(t_i,z)dz=1.
\end{split}
\end{equation}
\end{LEMME}
\begin{proof}
The exponential control is derived similarly to the diffusive case, and comes from the specific form of the parametrix expansion used to analyze the error.
The control in small time can also be derived from this representation, similarly to what occurs in \cite{kona:mamm:09}.  To make this last point clear we give some details in Appendix \ref{ASYMP_SMALL_TIMES}.
\end{proof}
Lemma \ref{LEMME_S_SENS_S} now follows from Lemma \ref{LE_LEMME_SPAT_STABLE} and equation \eqref{EQ_GROS_LEMME_STABLE} in Lemma \ref{GROS_LEMME}.

\section{Perspectives}
\label{PERP}
We did not consider in this work the case of a general Markov Chain approximating stable driven SDEs. To do so the first step would consist in establishing a LLT for sums of i.i.d random variables belonging to the domain of attraction of a stable law. Such results have been thoroughly studied in the scalar case, see e.g. Mitalauskas and Statuljavi{\v{c}}jus \cite{mita:stat:76} or the monograph by Christoph and Wolf \cite{chri:wolf:92}. We also refer to the PhD. of Squartini  \cite{squa:PhD:07} and to \cite{molc:petr:squa:07} (Cauchy case), for recent developments. We believe those results can be extended to the multidimensional case and would also \textit{transmit} to a Markov Chain approximation through a continuity technique like the parametrix. 
Now such limit theorems could also allow to consider more general \textit{fractional like} time derivatives that would be associated with the inverse of 
inhomogeneous stable subordinators $(S_t^+)_{t\ge 0} $ with generators
\begin{equation*}
L_t\phi(s)=\int_{\R_+^{*}} (\phi(s+u)-\phi(s)) g(t,u)\frac{du}{u^{1+\alpha}} . 
\end{equation*}
Eventually,  an analysis of generalized \textit{fractional like} derivatives involving as well the spatial variable like in Kolokoltsov's work, see e.g. Proposition 4.4 in \cite{kolo:09}, would require to establish LLT for approximations of the couple $(S_t^+,X_t) $ with generator
\begin{equation*}
\begin{split}
{\mathcal L}_t\phi(s,x)=\int_{\R_+^{*}} (\phi(s+u,x)-\phi(s,x)) g(t,u,x)\frac{du}{u^{1+\alpha}}\\
+\int_{\R^d} \left(\phi(s,x+z)-\phi(s,x)\right) f(t,x,z)\frac{dz}{|z|^{d+\alpha}}. 
\end{split}
\end{equation*}

This case is much more delicate in the sense that the complete coupling breaks the independence in \eqref{REP_FK} between the spatial motion and the process associated with the fractional like time derivative. We anyhow believe that the arguments of \cite{kolo:09} can be adapted provided the LLT holds. This will concern further research. 
\appendix
\section{Derivation of the Error Bounds in Small and Large Time.}
\label{ASYMP_SMALL_TIMES}
We briefly explain in this section the bounds appearing in \eqref{DIFF_EUL_SCHEME}, \eqref{LLT_DIFF_MARG} and \eqref{CTR_DIFF_S_TIME} that are crucial for our analysis in order to balance the indicated explosive behaviours with the decays of the density of the inverse subordinator, see Lemma \ref{BOUNDS_P_Z_P_Z_H}. Actually, the only control that needs to be fully justified is \eqref{CTR_DIFF_S_TIME} in small time. The other bounds in small time are already established in the previously quoted papers. Also, the exponential bounds are in some sense \textit{classical}.
  
Let us consider the case of the Euler scheme associated with \eqref{SDE} first. The crucial point is that the densities $p(t_i,x,.), p^h_{{\rm Eul}}(t_i,x,.)$, of respectively  $X_{t_i}$ in \eqref{SDE} and its Euler scheme  $X_{t_i}^h $ in \eqref{SCHEME} starting at $x$, enjoy a \textit{parametrix} expansion, see again \cite{frie:64}, \cite{mcke:sing:67}, \cite{kona:mamm:00}, \cite{kona:mamm:02} for the diffusive case and \cite{kolo:00} and \cite{kona:meno:10} for the stable one.  
We proceed with the simplest example of a non degenerate diffusion, but the  results for the other cases can be derived similarly.
Consider the dynamics \eqref{SDE} under \A{A${}_D $}. The p.d.f $p(t,x,.)$ of $X_t^x $ exists for every $t>0$ and writes:
\begin{eqnarray}
\label{DEV_PARAM}
p(t,x,y)=\tilde p(t,x,y)+\sum_{i\ge 1 }\tilde p\otimes H^{(i)}(t,x,y),
\end{eqnarray} 
where for all $(t,x,y)\in \R_+^{*}\times \R^{2d},\ H(t,x,y)=(L-\tilde L)\tilde p(t,x,y) $ where $\tilde p(t,x,y) $ stands for the p.d.f. at point $y$ of the \textit{frozen} Gaussian or stable process $\tilde X_t^{x,y}=x+b(y)t+\sigma(y)Y_t $, with generator $\tilde L $. Note that the frozen density is always considered at the freezing point. In \eqref{DEV_PARAM} for two integrable functions $f,g:\R_+^{*}\times (\R^d)^2 \rightarrow \R$, we denote $f\otimes g(t,x,y):=\int_0^t du\int_{\R^d}f(u,x,z)g(t-u,z,y)dz $ and also for all $i\ge 1,\ H^{(i)}:=H^{(i-1)}\otimes H,\ f\otimes H^{(0)}=f
$.

\subsection{Diffusion Case.}
\label{APP_DIFF}
It is  known (see \cite{frie:64}, \cite{kona:mamm:02}) that, under \A{A${}_D$}, there exist $(c_1,c_2):=(c_1,c_2)($\A{A${}_D $})$\ge 1 $ s.t. for all $i\ge 1, t>0 $,
\begin{equation}
\label{CTR_KER}
|\tilde p\otimes H^{(i)}(t,x,y)|\le \frac{(c_1 t)^{i/2}}{\Gamma(i-1/2)}g_{c_2}(t,x-y).
\end{equation}
From equations \eqref{CTR_KER} and \eqref{DEV_PARAM}, it is easily seen that there exists $(c_1,c_2):=(c_1,c_2)($\A{A${}_D $}) s.t.
\begin{equation}
\label{Borne_EXP}
p(t,x,y)\le c_1 \exp(c_1 t^{1/2})g_{c_2}(t,x-y).
\end{equation}
This explains the exponential control. This approach is due to McKean and Singer \cite{mcke:sing:67}. It can be easily transposed to the approximation schemes. The p.d.f. of Euler approximations enjoy similar properties as the one of the initial SDE (see the quoted references).

The control in small time comes from the explicit form of the error decomposition  which is similar to \eqref{DEV_PARAM}. Namely, one has (see \cite{kona:mamm:02} in the considered case)
\begin{equation}
\begin{split}
(p-p^h_{{\rm Eul}})(t_i,x,y)
&=\frac {h}2\left( p\otimes (L^2- L_*^2) p\right)(t_i,x,y)+h^2 R(t_i,x,y),
\end{split}
\label{ERR_EXP_EUL}
\end{equation}
where the remainder term satisfies as well $|R(t_i,x,y)|\le  c_1 \exp(c_1t_i^{1/2})t_i^{-1/2} g_{c_2}(t_i,x-y)  $. 
Also, we denote for $i\in \{1,2\} $, $L_* ^{i} \phi(x):=(L_\xi^i \phi(x))|_{\xi=x},\ L_\xi\phi(x)=\langle b(\xi),\nabla \phi(x)\rangle +\frac 12{\rm Tr}(a(\xi)D_x^2\phi(x))$. Observe that $L\phi(x)=L_*\phi(x) $, but more generally the operators do not coincide anymore when iterated. Precisely considering $d=1$ to alleviate the notations, we get for all $t\in (0,t_i) $:
\begin{eqnarray} 
(L^2-L_*^2)p(t,x,y)=\nonumber\\
\{b(x)b'(x)+\frac 12 a(x)b''(x) \} \partial_x p(t,x,y)\nonumber\\
+\{\frac12 b(x)a'(x)+a(x)b'(x)+\frac14 a(x)a''(x) \}\partial_x^2 p(t,x,y)\nonumber\\
+\frac 12 a(x)a'(x)\partial_x^3 p(t,x,y).\label{MAIN_TERM_DEV}
\end{eqnarray}
Now, the controls on the density and its derivatives imply under the smoothness assumptions in \A{A${}_D$} similarly to 
\eqref{Borne_EXP} that there exists $\bar c:=\bar c($\A{A${}_D$}) s.t. for all multi-indexes $\alpha,\beta, |\alpha|+|\beta|\le 3 $ for all $t>0$:
\begin{equation}
\label{CTR_DER}
|\partial_x^\alpha\partial_y^\beta\partial p(t,x,y)| \le \frac{\bar c}{t^{\frac{|\alpha|+|\beta|}2}}\exp(\bar ct^{1/2}) g_{c_2}(t,x-y).
\end{equation}
We refer to Friedman \cite{frie:64} for details. Hence, the most \textit{singular} term in \eqref{MAIN_TERM_DEV} is the last one.
Precisely,
\begin{eqnarray}
|a(x)a'(x)\partial_x^3 p(t,x,y)|&\le& \frac{c}{t^{3/2}} \exp(\bar ct^{1/2}) g_{c_2}(t,x-y).
\label{worst_term}
\end{eqnarray}
Now, recalling from \eqref{ERR_EXP_EUL} that this control needs to be plugged in the time-space convolution $p\otimes [(L^2-L_*^2)]p(t_i,x,y)=\int_0^{t_i} ds \int_{\R^d}p(s,x,z) (L^2-L_*^2)p(t_i-s,z,y)dz$,
two cases can occur: if $s\in [0,t_i/2] $ the control in \eqref{worst_term} is not singular and gives the stated bound of $t_i^{-1/2}$  once integrated in time. On the other hand when $s\in [t_i/2,t_i] $, some integration by parts need to be performed for the singular term yielding a third order spatial derivative in $z$ of $ p(s,x,z)$. Thanks to \eqref{CTR_DER}, which gives $ |\partial_z^3 p(s,x,z)| \le \frac{\bar c}{s^{\frac{3}2}}\exp(\bar cs^{1/2}) g_{c_2}(s,x-z)$, and the integration in time, we get again a  time singularity in $t_i^{-1/2} $ for the main contribution in \eqref{ERR_EXP_EUL}. The remainder can be handled similarly (see \cite{kona:mamm:02}). This analysis can be performed as well for the Markov Chain approximation, see \cite{kona:mamm:09}.

\subsection{Strictly Stable Case.}
The expansion \eqref{DEV_PARAM} also holds under \A{A${}_S$}, see \cite{kolo:00} from which we have
\begin{eqnarray*}
|\tilde p\otimes H^{(i)}(t,x,y)|\le \frac{(c_1 t)^{i\omega}}{\Gamma(i-\omega)}p_{S}(t,x-y), \ \omega:=\left(\frac 1\alpha \wedge 1\right)  \in (1/2,1].
\end{eqnarray*}
This therefore gives the indicated exponential bound  $p(t,x,y)\le c\exp(ct^\omega)p_S(t,x-y) $ for $c:=c($\A{A${}_S$}).
Now, let us denote by $f$ the spherical density of the measure $\nu $ (see assumption \A{ND}). We can rewrite: 
\begin{eqnarray*}
L\varphi(x)=L_*\varphi(x)=
\langle b(x),\nabla_x \varphi(x)\rangle\\
+\int_{\R^d}\{ \varphi(x+\sigma(x)z)-\varphi(x) -\langle \nabla_x \varphi(x), \sigma(x)z\rangle \I_{|\sigma(x) z|\le 1}\}f\left(\frac{z}{|z|}\right)\frac{dz}{|z|^{d+\alpha}} \\
=\langle b(x),\nabla_x \varphi(x)\rangle+\int_{\R^d}\{ \varphi(x+z)-\varphi(x) -\langle \nabla_x \varphi(x), z\rangle \I_{|z|\le 1}\} 
\Theta(x,z)
dz,
\end{eqnarray*}
where we denoted 
for all $\zeta\in \R^d $,
\begin{eqnarray}
\label{DEF_THETA_BIS}
\Theta(\zeta,z):=\frac{f\left(\frac{\sigma^{-1}(\zeta)z}{|\sigma^{-1}(\zeta)z|} \right)}{|\sigma^{-1}(\zeta)z|^{d+\alpha}\det(\sigma(\zeta))}.
\end{eqnarray}
With the notations of the previous section (considering $\alpha<1$ to avoid the cut-off and alleviate the notations) we get:
\begin{eqnarray*}
L^2p(t,x,y)&=&\langle b(x), \nabla_x  b(x) \nabla_xp(t,x,y)\rangle
+\langle b(x),   D_x^2p(t,x,y) b(x) \rangle\\
&&+\langle b(x),\int_{\R^d} (\nabla_x p(t,x+z,y)-\nabla_x p(t,x,y))\Theta(x,z)dz\rangle \\
&&+\langle b(x),\int_{\R^d} ( p(t,x+z,y)- p(t,x,y))\nabla_x \Theta(x,z)dz\rangle \\
&&+\int_{\R^d}\bigg\{[\int_{\R^d}\{p(t,x+z'+z,y)-p(t,x+z',y)    \}\Theta(x+z',z)dz]\\
&& -[\int_{\R^d}\{p(t,x+z,y)-p(t,x,y)    \}\Theta(x,z)dz]  \bigg\} \Theta(x,z')dz',\\
L_*^2p(t,x,y)&=&\bigg[\langle b(x),    D_x^2p(t,x,y) b(\xi)\rangle\\
&&+\langle b(x),\int_{\R^d} (\nabla_x p(t,x+z,y)-\nabla_x p(t,x,y))\Theta(\xi,z)dz\rangle \\
&&+\int_{\R^d}\bigg\{[\int_{\R^d}\{p(t,x+z'+z,y)-p(t,x+z',y)    \}\Theta(\xi,z)dz]\\
&& -\left.[\int_{\R^d}\{p(t,x+z,y)-p(t,x,y)    \}\Theta(x,z)dz]  \bigg\} \Theta(\xi,z')dz' \bigg]\right|_{\xi=x}.
\end{eqnarray*}
Thus, the difference writes:
\begin{eqnarray*}
(L^2-L_*^2)p(t,x,y)&=&\langle b(x), \nabla_x  b(x) \nabla_xp(t,x,y)\rangle\\
&&+\langle b(x),\int_{\R^d} ( p(t,x+z,y)- p(t,x,y))\nabla_x \Theta(x,z)dz\rangle\\
&&+\int_{\R^d}\bigg\{[\int_{\R^d}\{p(t,x+z'+z,y)-p(t,x+z',y)    \}\\
&&\times \{ \Theta(x+z',z)-\Theta(x,z)\}dz]\times \Theta(x,z')dz'.
\end{eqnarray*}
We mention that, even though we assumed $b=0$ for $\alpha<1 $ we kept the explicit dependence on $b$ in the above computations for the sake of completeness. Let us now focus on the last term in the above equation:
\begin{eqnarray*}
D(t,x,y):=\int_{\R^d} \bigg\{ \int_{\R^d} (p(t,x+z+z',y)-p(t,x+z',y) \\
 \times 
\{ \Theta(x+z',z)-\Theta(x,z)\} 
dz
\bigg\}
\times \Theta(x,z')dz',
 \end{eqnarray*}
which is the only contribution if $\alpha<1 $ and which would be the most singular for $\alpha\ge 1 $.
 The smoothness and non-degeneracy assumptions \A{S}, \A{UE}, \A{ND} on the coefficients yield:
\begin{eqnarray*}
|\nabla_x\Theta(x,z)|&\le& \frac{c}{|z|^{d+\alpha}},\\
|D(t,x,y)|&\le &c \int_{\R^d}  \int_{\R^d} |p(t,x+z+z',y)-p(t,x+z',y)| \\
&&\times \frac{\I_{|z'|>\frac 14 t^{1/\alpha}} +|z'|\I_{|z'|\le \frac 14 t^{1/\alpha}}}{|z|^{d+\alpha}}\times  \frac{1}{|z'|^{d+\alpha}}dz' dz, 
\end{eqnarray*}
where we have truncated with respect to the characteristic time-scale $t^{1/\alpha} $ (up to a constant) in the variable $z'$. Indeed, if 
$|z'|\le \frac 14 t^{1/\alpha}$ we perform a Taylor expansion and exploit the previous bound on $\nabla_x \Theta $ whereas if $|z'|> \frac 14 t^{1/\alpha} $ we simply use the uniform bound of $\Theta $ deriving from its definition in \eqref{DEF_THETA_BIS} and the non degeneracy assumptions.
Write now:
\begin{eqnarray*}
|D(t,x,y)|&\le& c\int_{|z'|\le \frac 14 t^{1/\alpha}}  \int_{\R^d} |p(t,x+z+z',y)-p(t,x+z',y)| \\
&&\times \frac{|z'|}{|z'|^{d+\alpha}}\times  \frac{1}{|z|^{d+\alpha}}dz' dz\\
&&+c\int_{|z'|>\frac 14 t^{1/\alpha}}  \int_{\R^d} |p(t,x+z+z',y)-p(t,x+z',y)| \\
&&\times \frac{1}{|z'|^{d+\alpha}}\times  \frac{1}{|z|^{d+\alpha}}dz' dz:=(D_1+D_2)(t,x,y).
\end{eqnarray*}
Let us treat those two terms separately. 
Recalling from \cite{kolo:00} and the above bounds that 
$|\nabla_x  p(t,x,y)|\le \frac{c\exp(ct^\omega)}{t^{1/\alpha}} p_S(t,x-y)$ we derive:
\begin{eqnarray*}
D_1(t,x,y)&\le &  \frac{c\exp(ct^\omega)}{t^{1/\alpha}} \int_{|z'|\le \frac 14 t^{1/\alpha},|z|\le \frac14 t^{1/\alpha}}  p_S(t,x+z'+\theta z-y) |z|\\
&&\times \frac{|z'|}{|z'|^{d+\alpha}}\times  \frac{1}{|z|^{d+\alpha}}dz' dz\\
&&+ \int_{|z'|\le \frac 14 t^{1/\alpha},|z|>\frac 14 t^{1/\alpha}}  |p(t,x+z+z',y)-p(t,x+z',y)| \\
&&\times \frac{|z'|}{|z'|^{d+\alpha}}\times  \frac{1}{|z|^{d+\alpha}}dz' dz:=(D_{11}+D_{12})(t,x,y),
\end{eqnarray*}
for some $\theta:=\theta(x,z,z',y)\in [0,1] $ in $D_{11}$. In that contribution,  since both $z,z'$ are small w.r.t. the characteristic time $t^{1/\alpha} $, we have:
\begin{eqnarray}
p_S(t,x+z'+\theta z-y)&\le& \frac{c}{t^{d/\alpha}}\frac{1}{(1+\frac{|x+z'+\theta z-y|}{t^{1/\alpha}})^{d+\alpha}}\le\frac{c}{t^{d/\alpha}}\frac{1}{(\frac 12+\frac{|x-y|}{t^{1/\alpha}})^{d+\alpha}}\nonumber\\
&\le & c 2^{d+\alpha}p_S(t,x-y).\label{FIRST_EQUIV_STABLE}
\end{eqnarray}
Hence:
\begin{eqnarray*}
D_{11}(t,x,y)&\le & \frac{c\exp(ct^\omega)}{t^{1/\alpha}} p_S(t,x-y)\left(\int_{r\le \frac14 t^{1/\alpha}}\frac{dr}{r^\alpha}\right)^{2}\le \frac{c\exp(ct^\omega)}{t^{2-\frac1\alpha}}p_S(t,x-y).
\end{eqnarray*}
For the contribution $  |p(t,x+z+z',y)-p(t,x+z',y)|$ in $D_{12}(t,x,y)$, since $|z|$ can be large, we write directly, recalling as well that $|z'| $ is small
and proceeding as in \eqref{FIRST_EQUIV_STABLE}:
\begin{eqnarray}
|p(t,x+z+z',y)-p(t,x+z',y)|&\le& |p(t,x+z+z',y)|+|p(t,x+z',y)|\nonumber \\
&\le &c (p_S(t,x+z-y)+p_S(t,x-y)).\label{CTR_PZZP}
\end{eqnarray}
Observe that if $|x-y|\le t^{1/\alpha}$, i.e. the diagonal regime holds for $p_S(t,x-y) $, we can then  use the global upper-bound $ (p_S(t,x+z-y)+p_S(t,x-y))\le \frac{c}{t^{d/\alpha}}$, in the control
\eqref{CTR_PZZP} so that
\begin{eqnarray*}
D_{12}(t,x,y)&\le &c\exp(ct^\omega) p_S(t,x-y)\int_{|z'|\le \frac 14t^{1/\alpha}}\frac{dz'|z'|}{|z'|^{d+\alpha}}\int_{|z|> \frac 14t^{1/\alpha}}\frac{dz}{|z|^{d+\alpha}}\\
&\le & c\exp(ct^\omega) p_S(t,x-y)t^{-2+\frac 1 \alpha}.
\end{eqnarray*}
If now $|x-y|> t^{1/\alpha} $, we have  for all given $\varepsilon\in (0,1) $, $p_S(t,x+z-y)\le c_\varepsilon p_S(t,x-y) $ if $z\not \in B(x-y,\varepsilon |x-y|) $.  Indeed, $|x-y+z|\ge\big| |x-y|-|z|\big| \ge \varepsilon |x-y|$. On the other hand, if $z\in B(x-y,\varepsilon|x-y|),\ |z|\ge (1-\varepsilon)|x-y| $ and $|z|^{-(d+\alpha)}\le ((1-\varepsilon)|x-y|)^{-(d+\alpha)}\le c_\varepsilon  t^{-1}p_S(t,x-y)$, up to a possible modification of $c_\varepsilon$.
Thus,
\begin{eqnarray*}
D_{12}(t,x,y)\le c\exp(ct^\omega)\bigg\{t^{-2+\frac 1\alpha}p_S(t,x-y)+\\
c_\varepsilon p_S(t,x-y)\int_{B(x-y),\varepsilon |x-y|)^c\cap\{|z|>\frac 14 t^{1/\alpha} \}}\frac{dz}{|z|^{d+\alpha}}\int_{|z'|\le \frac 14 t^{1/\alpha}} \frac{|z'|dz'}{|z'|^{d+\alpha}} +\\
\frac{1}{\{(1-\varepsilon)|x-y|\}^{d+\alpha}}\int_{B(x-y),\varepsilon |x-y|) 
}p_S(t,x+z-y)dz\int_{|z'|\le \frac 14t^{1/\alpha}} \frac{|z'|dz'}{|z'|^{d+\alpha}}
\bigg\}\\
\le c\exp(ct^\omega)p_S(t,x-y)t^{-2+\frac 1\alpha}.
\end{eqnarray*}
This proves that:
\begin{equation}
\label{CTR_D1}
D_1(t,x,y)\le c\exp(ct^\omega)p_S(t,x-y)t^{-2+\frac 1\alpha}.
\end{equation}
Let us observe that we have some continuity for the singularity w.r.t. the stability index, w.r.t. to the diffusive case, as far as the small jumps are concerned (see equation \eqref{worst_term}). The large jumps deteriorate the singularity.
Precisely,
\begin{eqnarray*}
|D_2(t,x,y)|\le \int_{|z'|> \frac 14 t^{1/\alpha}}\frac{dz'}{|z'|^{d+\alpha}}\\
\times \int_{\R^{d}}|p(t,x+z'+z,y)-p(t,x+z',y)| \frac{dz}{|z|^{d+\alpha}}.
\end{eqnarray*}
From the previous arguments, splitting again the small and large jumps in the $z$ variable w.r.t. the characteristic time scale $t^{1/\alpha} $, we get:
\begin{eqnarray*}
\int_{\R^{d}}|p(t,x+z'+z,y)-p(t,x+z',y)| \frac{dz}{|z|^{d+\alpha}}\le \frac ct\exp(ct^{\omega})p_S(t,x+z'-y).
\end{eqnarray*}
This therefore gives following the dichotomy used for the term $D_{12}(t,x,y)$:
\begin{eqnarray*}
|D_2(t,x,y)|\le \frac c{t^2}\exp(ct^{\omega})p_S(t,x-y),
\end{eqnarray*}
which together with \eqref{CTR_D1} indeed gives the bound
$|D(t,x,y)|\le \frac c{t^2}\exp(ct^{\omega})p_S(t,x-y) $.

Let us emphasize from \cite{kona:meno:10} that
 the expansion \eqref{ERR_EXP_EUL} also holds in the strictly stable case. The previous control then gives the result, similarly to the discussion in the previous paragraph if $s\in [0,t_i/2]$. For $s\in [t_i/2,t_i]$ we take the spatial adjoint and the associated contribution can be analyzed similarly. Eventually, this analysis extends to the remainder terms in \eqref{ERR_EXP_EUL}.

\section{Two-sided Heat Kernel Estimates}
\label{APP_HK}

We derive in this section two-sided estimates for the density of $X_{Z_T^\beta}$. Precisely we have the following Theorem.
\begin{THM}[Two-Sided Heat Kernel Bounds]
\label{HK_BOUNDS_FRAC}
\hspace*{2cm}
\begin{trivlist}
\item[-] \textbf{Diffusive Case:} Under \A{A${}_D $}, there exists $c:=c(\beta, $\A{A${}_D $})$\ge 1 $ s.t. 
for a given $T>0$ and for all $(x,y)\in (\R^d)^2 $, one has:
\begin{trivlist}
\item[$\bullet $] For $d=1$:
\begin{eqnarray}
\frac{c^{-1}}{T^{\beta/2}}\bigg \{ \exp(-cT^{\beta}) \exp\left(-c\frac{|x-y|^2}{T^\beta} \right)+\exp(-cT 
)\exp\left(-c\Big\{\frac{|x-y|^2}{T^\beta}\Big\}^{\frac{1}{2-\beta}} \right)\bigg\}\nonumber\\
\le p_{X_{Z_T^\beta}}(x,y) \le \nonumber\\
\frac{c}{T^{\beta/2}}\bigg\{\exp(cT^{\frac \beta 2})\exp\left(-c^{-1}\frac{|x-y|^2}{T^\beta} \right) +\exp(cT^{\frac \beta{1+\beta})}) \exp\left(-c^{-1}\Big\{\frac{|x-y|^2}{T^\beta}\Big\}^{\frac{1}{2-\beta} } \right)\bigg \}
.\nonumber \\
\end{eqnarray}
\item[$\bullet $] For $d= 2$ and $x\neq y $:
\begin{eqnarray}
\frac{c^{-1}}{T^{\beta}}\bigg \{ \exp(-cT^{\beta}) \exp\left(-c\frac{|x-y|^2}{T^\beta} \right)\big(|\log(c^{1/2}\frac{|x-y|}{T^{\beta/2}})| 
\I_{|x-y|\le c^{-1/2}T^{\beta/2}}
+1\big)\nonumber
\\+\exp(-cT
)\exp\left(-c\Big\{\frac{|x-y|^2}{T^\beta}\Big\}^{\frac{1}{2-\beta}} \right)\bigg\}\nonumber\\
\le p_{X_{Z_T^\beta}}(x,y) \le \nonumber\\
\frac{c}{T^{\beta}}\bigg\{\exp(cT^{\frac \beta 2})\exp\left(-c^{-1}\frac{|x-y|^2}{T^\beta} \right)\big( |\log(c^{-1/2}\frac{|x-y|}{T^{\beta/2}})|
\I_{|x-y|\le c^{1/2}T^{\beta/2}}
+1\big) \nonumber\\
+\exp(cT^{\frac \beta{1+\beta})}) \exp\left(-c^{-1}\Big\{\frac{|x-y|^2}{T^\beta}\Big\}^{\frac{1}{2-\beta} } \right)\bigg \}
.\nonumber \\
\label{HK_DIFF_2}.
\end{eqnarray}
\item[$\bullet $] For $d\ge 3$ and $x\neq y $:
\begin{eqnarray}
c^{-1}\bigg \{\frac{ \exp(-cT^{\beta})}{T^\beta|x-y|^{d-2}} \exp\left(-c\frac{|x-y|^2}{T^\beta} \right)
+\frac{\exp(-cT
)
}{T^{\beta d/2}}\exp\left(-c\Big\{\frac{|x-y|^2}{T^\beta}\Big\}^{\frac{1}{2-\beta}} \right)\bigg\}\nonumber\\
\le p_{X_{Z_T^\beta}}(x,y) \le \nonumber\\
c\bigg\{\frac{\exp(cT^{\frac \beta 2})}{T^\beta |x-y|^{d-2}}\exp\left(-c^{-1}\frac{|x-y|^2}{T^\beta} \right) 
+\frac{\exp(cT^{\frac \beta{1+\beta})})}{T^{\beta d/2}} \exp\left(-c^{-1}\Big\{\frac{|x-y|^2}{T^\beta}\Big\}^{\frac{1}{2-\beta} } \right)\bigg \}
. 
\nonumber \\
 \label{HK_DIFF_GE3}
\end{eqnarray}
\item[-] \textbf{Strictly Stable Case:} Under \A{A${}_S $}, there exists $c:=c(\beta, $\A{A${}_S $})$\ge 1 $ s.t. 
for $d\ge 2$ and  $\alpha \in (0,2) $ or $d=1, \alpha\le 1$, a given $T>0$ and for all $(x,y)\in (\R^d)^2, x\neq y $, one has:
\begin{eqnarray}
c^{-1}\bigg\{\frac{\exp(-cT^{ \beta }) }{T^\beta |x-y|^{d-\alpha}}\I_{|x-y|\le T^{\beta/\alpha}}+ 
\frac{1}{T^{\beta d/\alpha}}\frac{\exp(-cT^\beta)}{\Big(1+\frac{|x-y|}{T^{\beta/\alpha}}\Big)^{d+\alpha}}\bigg \}\nonumber
\le p_{X_{Z_T^\beta}}(x,y) \le \nonumber\\
c\bigg\{\frac{\exp(cT^{ \beta \omega}) }{T^\beta |x-y|^{d-\alpha}}\I_{|x-y|\le T^{\beta/\alpha}}+ 
\frac{1}{T^{\beta d/\alpha}}\frac{\exp(cT^{\frac{\beta\omega}{1-\omega(1-\beta)}})}{\Big(1+\frac{|x-y|}{T^{\beta/\alpha}}\Big)^{d+\alpha}}\bigg \},\ \omega:=\frac 1\alpha \wedge 1
.\nonumber \\
\end{eqnarray}
If $d=1, \alpha>1  $ we get for all $(x,y)\in \R^2 $:
\begin{eqnarray}
c^{-1}\bigg\{
\frac{1}{T^{\beta d/\alpha}}\frac{\exp(-cT^\beta)}{\Big(1+\frac{|x-y|}{T^{\beta/\alpha}}\Big)^{d+\alpha}}\bigg \}\nonumber
\le p_{X_{Z_T^\beta}}(x,y) \le 
c\bigg\{
\frac{1}{T^{\beta d/\alpha}}\frac{\exp(cT^{\frac{\beta\omega}{1-\omega(1-\beta)}})}{\Big(1+\frac{|x-y|}{T^{\beta/\alpha}}\Big)^{d+\alpha}}\bigg \},\ \omega:=\frac 1\alpha
.\nonumber \\
\end{eqnarray}
\end{trivlist}

\end{trivlist}
\end{THM}

\begin{proof}

We first mention that the arguments in the proof of Lemma \ref{BOUNDS_P_Z_P_Z_H} also give that, a lower bound homogeneous to \eqref{CTR_BOUNDS_P_Z_P_Z_H} holds for $p_{Z^\beta}(T,u) $.
Precisely:
\begin{equation}
\label{CTR_P_Z_LB}
\frac{c_\beta^{-1}}{T^\beta}\exp\left( -c_{\beta}\left\{\frac{u}{T^\beta} \right\}^{1/(1-\beta)}\right)\le p_{Z^\beta}(T,u).
\end{equation}

\begin{trivlist}
\item[-] \textbf{Diffusive case.}
Let us concentrate on the lower bounds, since the upper bounds are already derived in Corollary \ref{LEMME_QUI_FAIT_A_LA_KOCHUBEI}, and on the two-sided bounds for the two dimensional case. Note that we obtain in equations \eqref{HK_DIFF_2}, \eqref{HK_DIFF_GE3}, for dimensions $d\ge 2$, the same singularities that we observe for the Poisson kernel of the associated dimension.

We focus below on the lower bound for $d\ge 3$ and the two-sided bounds for $d= 2$. The other controls can be derived from arguments similar  to those developed below. One of the key points for the proof is the lower bound for the density of the spatial motion. Namely a lower bound homogeneous to the one in Lemma \ref{HKUSUEL} holds for the density of the spatial motion. It is classical, from chaining arguments, see e.g. Chapter 7 in Bass \cite{bass:97}, to derive from the parametrix expansions presented above that there exists $c:=c($\A{A${}_D $}) $\ge 1$ s.t. for all $(x,y)\in (\R^d)^2 $:
\begin{equation}
\label{LB_DIFF}
p(t,x,y)\ge c^{-1}\frac{\exp(-ct)}{t^{d/2}}\exp\left(- c\frac{|x-y|^2}{t}\right).
\end{equation}

\textit{Lower bound for $d\ge 3 $.}
We now start from \eqref{DENS_GLOB_MARKOV}, \eqref{CTR_P_Z_LB} and \eqref{LB_DIFF} to derive:
\begin{eqnarray}
\label{DECOUP_M1M2_DIFF}
p_{X_{Z_T^\beta}}(x,y)\ge \frac{c^{-1}}{T^\beta}
\left \{\exp(-cT^\beta)\int_{0}^{T^\beta}\frac{du}{u^{d/2}}\exp(-c\frac{|x-y|^2}{u})  \right. \nonumber\\
\left.+\int_{T^\beta}^{+\infty} \frac{du}{u^{d/2}}\exp(-c\frac{|x-y|^2}{u})\exp(-cu)\exp\left(-c\left[\frac{u}{T^\beta}\right]^{1/(1-\beta)}\right)\right\}\nonumber\\
:=(m_1+m_2)(T^\beta,x-y).
\end{eqnarray}
To control $m_1(T^\beta,x-y) $ we consider again the previous, diagonal/off-diagonal dichotomy:
\begin{trivlist} 
\item[-] For $ |x-y|/T^{\beta/2}\le 1$ (diagonal regime) write:
\begin{eqnarray*}
m_1(T^\beta,x-y)&\ge& \frac{c^{-1}\exp(-cT^\beta)}{T^{\beta}}\int_{|x-y|^2/2}^{|x-y|^2} \frac{du}{u^{d/2}}\\
&\ge &\frac{c^{-1}\exp(-cT^\beta)}{T^{\beta}}\int_{|x-y|^2/2}^{|x-y|^2} \frac{du}{|x-y|^d}\ge \frac{c^{-1}\exp(-cT^\beta)}{T^\beta|x-y|^{d-2}}.
\end{eqnarray*}
\item[-] For $ |x-y|/T^{\beta/2}>1$ (off-diagonal regime) write:
\begin{eqnarray*}
m_1(T^\beta,x-y)&\ge &\frac{c^{-1}\exp(-cT^\beta)}{T^\beta |x-y|^{d-2}}\int_{T^{\beta}/2}^{T^\beta} \frac{du}{u} \left(\frac{|x-y|}{u^{1/2}}\right)^{d-2}\exp\left(-c\frac{|x-y|^2}{u} \right)\\
&\ge & \frac{c^{-1}\exp(-cT^\beta)}{T^\beta |x-y|^{d-2}}\exp\left (-c\frac{|x-y|^2}{T^\beta} \right) \int_{T^\beta/2}^{T^\beta} \frac{du}{u}\\
&\ge & \frac{c^{-1}\exp(-cT^\beta)}{T^\beta |x-y|^{d-2}}\exp\left (-c\frac{|x-y|^2}{T^\beta} \right).
\end{eqnarray*}
\end{trivlist}
We have thus established that:
\begin{equation}
\label{CTR_M1} 
m_1(T^\beta,x-y)\ge \frac{c^{-1}\exp(-cT^\beta)}{T^\beta |x-y|^{d-2}}\exp\left (-c\frac{|x-y|^2}{T^\beta} \right).
\end{equation}

\end{trivlist}
On the other hand for the contribution $m_2(T^\beta,x-y) $ in \eqref{DECOUP_M1M2_DIFF} we get:
\begin{trivlist}
\item[-] For $ |x-y|/T^{\beta/2}\le 1$ (diagonal regime) write:
\begin{eqnarray*}
m_2(T^\beta,x-y)\ge \frac{c^{-1}\exp(-cT^\beta)}{T^{\beta(1+d/2)}}\int_{T^\beta}^{2T^{\beta}} du \exp\left(-c\left[\frac{u}{T^\beta} \right]^{1/(1-\beta)}\right) \ge \frac{c^{-1}\exp(-cT^\beta)}{T^{\beta d/2}}.
\end{eqnarray*}
\item[-] For $ |x-y|/T^{\beta/2}> 1$ (off-diagonal regime), write first from Young's inequality,
$$u\le c_\beta \big [(\frac{u}{T^\beta})^{1/(1-\beta)}+(T^{\beta})^{1/\beta}  \big].$$
Computations similar to those in Lemma \ref{GROS_LEMME} then yield:
\begin{eqnarray*}
m_2(T^\beta,x-y)\ge 
\frac{c^{-1}\exp(-cT
)}{T^{\beta}}\\
\times \int_{ [|x-y|^2T^{\frac\beta{1-\beta}}]^{\frac{1-\beta}{2-\beta}}}^{ 2[|x-y|^2T^{\frac{\beta}{1-\beta}}]^{\frac{1-\beta}{2-\beta}}} \frac{du}{u^{d/2}}\exp\left(-c\frac{|x-y|^2}{u} \right) \exp\left(-c\left[\frac{u}{T^\beta} \right]^{1/(1-\beta)}\right) \\
\ge \frac{c^{-1}\exp(-cT
)}{T^{\beta}} \exp\left( -c\left[\frac{|x-y|^2}{T^\beta} \right]^{\frac{1}{2-\beta}}\right)\frac{1}{[|x-y|^2T^{\frac{\beta}{1-\beta}}]^{\frac{1-\beta}{2-\beta}(d/2-1)}}\\
\ge \frac{c^{-1}\exp(-cT
)}{T^{\beta}} \exp\left( -c\left[\frac{|x-y|^2}{T^\beta} \right]^{\frac{1}{2-\beta}}\right)\frac{1}{[T^{\beta (1+\frac{1}{1-\beta})}]^{\frac{1-\beta}{2-\beta}(d/2-1)}}\\
\ge\frac{c^{-1}\exp(-cT
)}{T^{\beta d/2}} \exp\left( -c\left[\frac{|x-y|^2}{T^\beta} \right]^{\frac{1}{2-\beta}}\right),
\end{eqnarray*}
recalling that $\forall z\ge 1, \ \exp(-z^{1/(2-\beta)})z^{-\frac{1-\beta}{2-\beta}(d/2-1)} \ge c^{-1}\exp(-cz^{1/(2-\beta)}) $ for the last but one inequality.
We have thus proved in all cases:
\begin{equation*}
m_2(T^\beta,x-y)\ge \frac{c^{-1}\exp(-cT
)}{T^{\beta d/2}} \exp\left( -c\left[\frac{|x-y|^2}{T^\beta} \right]^{\frac{1}{2-\beta}}\right),
\end{equation*}
which together with \eqref{CTR_M1} and \eqref{DECOUP_M1M2_DIFF} completes the proof of the lower bound for $d\ge 3$.

\textit{Bounds for $d=2$.} The lower bound \eqref{DECOUP_M1M2_DIFF} and an homogeneous upper bound still hold, with obvious modifications of the constants. We focus on the contribution $m_1(T^\beta,x-y) $ yielding the additional spatial diagonal singularity. The term $m_2(T^\beta,x-y) $ can be analyzed as above. Let us write:
\begin{eqnarray*}
m_1(T^\beta,x-y)\ge \frac{c^{-1}\exp(-cT^\beta)}{T^\beta}\int_0^{T^\beta} \frac{du}{u}\exp\left(-c\frac{|x-y|^2}{u}\right).
\end{eqnarray*}
Assume first that $c^{1/2}|x-y|/T^{\beta/2}\le 1 $. Setting $v:=\exp\left(-c\frac{|x-y|^2}{u}\right) $ in the above integral we derive:
\begin{eqnarray*}
m_1(T^\beta,x-y)\ge \frac{c^{-1}\exp(-cT^\beta)}{T^\beta}\int_0^{\exp(-c\frac{|x-y|^2}{T^\beta})} \frac{dv}{-\log(v)}\\
\ge \frac{c^{-1}\exp(-cT^\beta)}{T^\beta}\left\{ -v\log(-\log(v))|_{0}^{\exp(-c\frac{|x-y|^2}{T^\beta})}+\int_0^{\exp(-c\frac{|x-y|^2}{T^\beta})} \log(-\log(v))dv\right\}\\
\ge \frac{c^{-1}\exp(-cT^\beta)}{T^\beta} \Big\{\exp(-c\frac{|x-y|^2}{T^\beta}) \big(-\log(c\frac{|x-y|^2}{T^\beta}) \big)\\
+\int_{e^{-1}}^{\exp(-c\frac{|x-y|^2}{T^\beta})}\log(-\log(v)) dv\Big\}\\
\ge \frac{c^{-1}\exp(-cT^\beta)}{T^\beta} \Big\{\exp(-c\frac{|x-y|^2}{T^\beta}) \big(-\log(c\frac{|x-y|^2}{T^\beta}) \big)\\
+\log(c\frac{|x-y|^2}{T^\beta}) (\exp(-c\frac{|x-y|^2}{T^\beta})-e^{-1})\Big\}.
\end{eqnarray*}
Hence:
\begin{eqnarray*}
m_1(T^\beta,x-y)&\ge& \frac{c^{-1}\exp(-cT^\beta)e^{-1}}{T^\beta}\big(-\log(c\frac{|x-y|^2}{T^\beta}) \big)\\
&\ge &\frac{c^{-1}e^{-1}\exp(-cT^\beta)\exp\left(-c\frac{|x-y|^2}{T^\beta} \right)}{T^\beta}|\log(c\frac{|x-y|^2}{T^\beta})| .
\end{eqnarray*}
The upper-bound could be derived similarly.

If now $c^{1/2}|x-y|/T^{\beta/2}> 1  $, we write:
\begin{eqnarray*}
m_1(T^\beta,x-y)&\ge &\frac{c^{-1}\exp(-cT^\beta)\exp(-c\frac{|x-y|^2}{T^\beta})}{T^\beta}\int_{T^\beta/2}^{T^\beta} \frac{du}{u}\\
&\ge&  \frac{c^{-1}\exp(-cT^\beta)\exp(-c\frac{|x-y|^2}{T^\beta})}{T^\beta}.
\end{eqnarray*}
The previous bounds give the result.
\item[-]{\textbf{Strictly Stable Case.}} We focus here on the lower bound for $d\ge 2$. The other cases can be handled similarly.
From the lower bound in Kolkoltsov \cite{kolo:00} for the density and a chaining argument one derives that there exists $c:=c($\A{A${}_S $})$\ge 1 $ s.t. for all $t>0,\ (x,y)\in (\R^d)^2 $:
\begin{equation}
\label{LB_STAB}
p(t,x,y)\ge \frac{c^{-1}\exp(-ct)}{t^{d/\alpha}}\frac{1}{\left(1+\frac{|x-y|}{t^{1/\alpha}} \right)^{d+\alpha}}.
\end{equation}
From \eqref{DENS_GLOB_MARKOV}, \eqref{CTR_P_Z_LB} and \eqref{LB_STAB} we now get:
\begin{eqnarray}
\label{DECOUP_M1M2_DIFF_S}
p_{X_{Z_T^\beta}}(x,y)\ge \frac{c^{-1}}{T^\beta}
\left \{\exp(-cT^\beta)\int_{0}^{T^\beta}\frac{du}{u^{d/\alpha}} \frac{1}{\left(1+\frac{|x-y|}{u^{1/\alpha}} \right)^{d+\alpha}}  \right. \nonumber\\
\left.+\int_{T^\beta}^{+\infty} \frac{du}{u^{d/\alpha}}\frac{1}{\left(1+\frac{|x-y|}{u^{1/\alpha}} \right)^{d+\alpha}}\exp(-cu)\exp\left(-c\left[\frac{u}{T^\beta}\right]^{1/(1-\beta)}\right)\right\}\nonumber\\
:=(m_1+m_2)(T^\beta,x-y).
\end{eqnarray}
Let us first control $m_1(T^\beta,x-y) $ exploiting again the diagonal/off-diagonal dichotomy.
\begin{trivlist}
\item[-] If $|x-y|\le T^{\beta/\alpha} $ then
\begin{eqnarray*}
m_1(T^\beta,x-y)\ge \frac{c^{-1}\exp(-cT^\beta)}{T^\beta}\int_{|x-y|^{\alpha}/2}^{|x-y|^\alpha} \frac{du}{u^{d/\alpha}}
\ge \frac{c^{-1}\exp(-cT^\beta)}{T^\beta|x-y|^{d-\alpha}}.
\end{eqnarray*}
\item[-] If $|x-y|> T^{\beta/\alpha} $ we get:
\begin{eqnarray*}
m_1(T^\beta,x-y)\ge \frac{c^{-1}}{T^\beta}
\exp(-cT^\beta)\int_{T^{\beta}/2}^{T^\beta} du  \frac{u}{|x-y|^{d+\alpha}}\ge  \frac{c^{-1}T^\beta}{|x-y|^{d+\alpha}}
\exp(-cT^\beta).
\end{eqnarray*}
\end{trivlist}
We have thus established 
\begin{equation}
\label{CTR_M1_S}
m_1(T^\beta,x-y)\ge c^{-1}\exp(-cT^\beta)\Big\{ \frac{1}{T^\beta |x-y|^{d-\alpha}}\I_{|x-y|\le T^{\beta/\alpha}}+\frac{T^\beta}{|x-y|^{d+\alpha}}\I_{|x-y|>T^{\beta/\alpha}}\Big\}.
\end{equation}
Let us now turn to $m_2(T^\beta,x-y) $. We get:
\begin{trivlist}
\item[-] If  $|x-y|\le T^{\beta/\alpha} $,
\begin{eqnarray*}
m_2(T^\beta,x-y)\ge  \frac{c^{-1}}{T^\beta}
\exp(-cT^\beta)\int_{T^{\beta}}^{2T^\beta} \frac{du}{u^{d/\alpha}} \ge \frac{c^{-1}}{T^{d\beta/\alpha}}
\exp(-cT^\beta).
\end{eqnarray*}
\item[-] If  $|x-y|>T^{\beta/\alpha} $,
\begin{eqnarray*}
m_2(T^\beta,x-y)\ge  \frac{c^{-1}}{T^\beta|x-y|^{d+\alpha}}
\exp(-cT^\beta)\int_{T^{\beta}}^{2 T^\beta} u du \ge
\frac{c^{-1}T^\beta}{|x-y|^{d+\alpha}}
\exp(-cT^\beta).
\end{eqnarray*}
Plugging the above estimates and \eqref{CTR_M1_S} into \eqref{DECOUP_M1M2_DIFF_S} gives the result. 
\end{trivlist}
\end{trivlist}
\end{proof}

\begin{REM}
Let us emphasize  that the bounds of Theorem \ref{HK_BOUNDS_FRAC} would hold under the weaker assumptions that the coefficients $b,\sigma$ are measurable and s.t. $\sigma\sigma^* $ is H\"older continuous and $b$ is bounded. Indeed, the two-sided heat kernel bounds for the spatial motion hold in that case. We can for instance refer to Sheu \cite{sheu:91} in the diffusive case or to Huang \cite{huan:15} in the strictly stable case.

Let us point out as well that the two sided bounds hold in the diffusive case under those assumptions for the density of $X_{Z_T^{\beta,h}}^{h,{\rm Eul}} $ associated with the Euler scheme approximation  for the spatial motion. Again, the key estimate is a two-sided bound for the Euler scheme which can be found in Lemaire and Menozzi \cite{lema:meno:10}. In the strictly stable case, the upper bound holds under the indicated assumptions. This is a consequence of the parametrix expansion for the density of the scheme, see \cite{kona:meno:10}. The lower bound is more delicate to obtain since even in the diffusive case, the localization arguments which are standard for the SDE need to be carefully adapted for the scheme. 
\end{REM}

\begin{REM}
We conclude saying that the results in Theorem \ref{THM_DIFF} could be slightly improved in light of the sharp estimates of Theorem \ref{HK_BOUNDS_FRAC} for $d=2$. We did not exploit those controls in the presentation of the main results mainly for notational coherence and simplicity.
\end{REM}

\section*{Acknowledgments}
The article was prepared within the framework of a subsidy granted to the HSE by the Government of the Russian Federation for the implementation of the Global Competitiveness Program.

\bibliographystyle{alpha}
\bibliography{bibli}

\end{document}